\documentclass[11pt]{article}

\usepackage[T1]{fontenc}
\usepackage[utf8]{inputenc}

\usepackage[a4paper, margin=2cm]{geometry}
\usepackage{amsmath}
\usepackage{amssymb}
\usepackage{amsthm}
\usepackage{bbm}
\usepackage[title]{appendix}
\usepackage{soul,xcolor}
\usepackage{leftindex}

\newcommand{\dd}{\mathrm{d}}
\newcommand{\ddiv}{\mathrm{div\,}}

\newcommand{\D}{\mathbb{D}}
\newcommand{\mS}{\mathbb{S}}

\newtheorem{theorem}{Theorem}

\newtheorem{definition}[theorem]{Definition}
\newtheorem{lemma}[theorem]{Lemma}
\newtheorem{proposition}[theorem]{Proposition}
\newtheorem{remark}[theorem]{Remark}

\newtheorem{tw}[theorem]{Theorem}

\newtheorem*{prop*}{Proposition}

\numberwithin{equation}{section}
\numberwithin{theorem}{section}

\title{Weak solutions to the Navier-Stokes equations for steady compressible non-Newtonian fluids}
\author{Cosmin Burtea\footnote{
Universit\'e de Paris and Sorbonne Universit\'e , CNRS, IMJ-PRG, F-75006 Paris, France; {\tt cosmin.burtea@imj-prg.}}, Maja Szlenk\footnote{University of Warsaw, Faculty of Mathematics, Informatics and Mechanics, Banacha 2, 02-097 Warsaw, Poland, e-mail address: {\tt m.szlenk@uw.edu.pl}}}
\date{\today}

\begin{document}
\setstcolor{red}

\maketitle

\begin{abstract}
    We prove the existence of weak solutions for the steady Navier-Stokes system for compressible non-Newtonian fluids on a bounded, two- or three-dimensional domain. Assuming the viscous stress tensor is monotone satisfying a power-law growth with power $r$ and the pressure is given by $\varrho^\gamma$, we construct a solution provided that $r>\frac{3d}{d+2}$ and $\gamma$ is sufficiently large, depending on the values of $r$. Additionally, we also show the existence for time-discretized model for Herschel-Bulkley fluids, where the viscosity has a singular part.
\end{abstract}


\section{Introduction}

We
investigate the existence of weak solutions for models for steady compressible
barotropic non-Newtonian fluids defined on a bounded $C^{2}$-domain $\Omega$.
The Navier-Stokes system of equations that govern the motion of such fluids
read
\begin{equation}
\left\{
\begin{aligned}
&\ddiv(\varrho u)  =0 \quad &\text{in } \Omega, \\
&\ddiv(\varrho u\otimes u)  -\ddiv\mathbb{S}(  \mathbb{D}u)  +\nabla p(\varrho)  =\varrho f+g \quad &\text{in } \Omega, \\
&u =0 \quad &\text{on }\partial\Omega,
\end{aligned}
\right.  \label{main}%
\end{equation}
where $\varrho=\varrho\left(  x\right)  \geq0$ represents the density of
fluid, $u=u\left(  x\right)  =\left(  u^{1}\left(  x\right)  ,...,u^{d}\left(
x\right)  \right)  \in\mathbb{R}^{d}$ represents the velocity of the fluid and the pressure is given by
\begin{equation}
p\left(  \varrho\right)  =a\varrho^{\gamma}\label{pressure}%
\end{equation}
for some $\gamma>1$, $a>0$ (without loss of generality we will assume $a=1$). The non-Newtonian aspect is encoded by the nonlinear form of the viscous stress tensor $\mS=\mS(\D u)$, where
\[ (\D u)_{i,j}=\frac{1}{2}(\partial_{x_j}u_i+\partial_{x_i}u_j) \]
is the symmetric gradient. We assume that $\mathbb{S}:\mathbb{R}_{sym}^{d\times d}\rightarrow\mathbb{R}_{sym}^{d\times
d}$ 
is a continuous function, which satisfies the following growth conditions 
\begin{equation}
\left\{
\begin{array}
[c]{r}%
\left\vert \mathbb{S}(A)\right\vert \leq C_{1}\left\vert A\right\vert
^{r-1},\\
C_{2}\left\vert A\right\vert ^{r}\leq S\left(  A\right)  :A
\end{array}
\right.  \label{growth}%
\end{equation}
for some $C_1,C_2>0$, and a monotonicity condition
\begin{equation}\label{monotonicity}
    C_3|A-B|^\ell \leq (\mS(A)-\mS(B)):(A-B), \quad C_3>0,
\end{equation}
for all symmetric matrices $A,B\in\mathbb{R}_{sym}^{d\times d}$ and some
$r\geq\ell>1$. We recall that the contraction between two
second order tensors is given by $A:B=a_{ij}b_{ij},$ where we use Einstein's
summation convention over repeated indices.

Typical examples that enter this framework include generic viscous stress
tensors of the form
\[
\mathbb{S}\left(  A\right)  =\mu\left(  \left\vert A\right\vert \right)
A+\lambda\left(  \left\vert \operatorname*{trace}A\right\vert \right)
\operatorname*{trace}A\mathbb{I\,}_{d}%
\]
where $\mathbb{I\,}_{d}$ is the identity tensor while $\mu$ and $\lambda$
which represent shear and bulk viscosities satisfying the growth conditions
\[
c_{1}|z|^{r-2}\leq\mu(|z|),\lambda(|z|)\leq c_{2}|z|^{r-2}.
\]
In particular, such choice of the viscosity parameters includes the power-law
fluids where the shear viscosity is given by
\[
\mu(|\mathbb{D}u|)=\mu_{0}|\mathbb{D}u|^{r-2},\lambda\left(  \left\vert
\operatorname{div}u\right\vert \right)  =\lambda_{0}\left\vert
\operatorname{div}u\right\vert ^{r-2},
\]
with $\mu_{0},\lambda_{0}>0$. We refer the readers to the paper
\cite{MalekRajagopal2010} for an extensive discussion on constitutive theory
for compressible fluids.

\subsection{A brief overview on the mathematical study of non-Newtonian
fluids}

\paragraph{Incompressible case.} The mathematical study of models for non-Newtonian fluids were thoroughly studied in the incompressible setting.
It was initiated by
O.A. Ladyzhenskaya \cite{Ladyzhenskaya1967,ladyzhenskaya1969mathematical}, who
studied the following system governing the evolution of the velocity field
$u=u\left(  t,x\right)  \in\mathbb{R}^{d}$ of an incompressible fluid filling
a bounded domain $\Omega\subset\mathbb{R}^{d}$:
\begin{equation}
\left\{
\begin{array}
[c]{r}%
\partial_{t}u+\operatorname{div}u\otimes u-\operatorname{div}\mathbb{S}+\nabla
p=f,\\
\operatorname{div}u=0,\\
u=0.
\end{array}%
\begin{array}
[c]{r}%
\text{in }\left(  0,\infty\right)  \times\Omega,\\
\text{in }\left(  0,\infty\right)  \times\Omega,\\
\text{on }\left(  0,\infty\right)  \times\partial\Omega.
\end{array}
\right.  \label{system}%
\end{equation}
In the above, the stress tensor satisfies the growth conditions (\ref{growth}) and for the monotonicity it was assumed that
\begin{equation}
c_{3}|A-B|^{r}+c_{4}|A-B|^{2}\leq(\mathbb{S}(A)-\mathbb{S}(B)):(A-B)
\label{mono_inc}%
\end{equation}
for
$A,B\in\mathbb{R}^{d\times d}$. In
\cite{Ladyzhenskaya1967,ladyzhenskaya1969mathematical} it was proven that there exist
weak solutions to $\left(  \text{\ref{system}}\right)  $ for the Dirichlet
boundary-value problem when $r\geq\frac{3d+2}{d+2}$. Moreover, the solutions are unique
if $r\geq\frac{d+2}{2}$ and $c_{4}>0$. Notice that the linear case of Newtonian
fluids, for which the growth conditions (\ref{growth}) are verified with $r=2$, is still an outstanding open problem
today. To cite P. Galdi \cite{galdi2011introduction} : \textquotedblleft%
\textit{what we do not know whether it is true for water, becomes certainly
true if we add to water a pinch of corn starch!}\textquotedblright.

Up to now, the early results of Ladyzhenskaya were improved in a series of papers in the following directions:

\underline{Existence :} the range of $r$ for which one can assure
existence of weak solutions has been extended, see for instance
\cite{Malek_et_al1993,wolf2007existence,Diening_et_al_2010}. As far as we are
aware, the lowest exponent for which existence of weak-solutions can be
ensured is
\[
r\geq\frac{2d}{d+2}%
\]
and was obtained in \cite{Diening_et_al_2010}, see also
\cite{frehse2003analysis} for the steady case.

\underline{Uniqueness :} the optimal requirement was obtained recently by
Bul{\'{\i}}{\v{c}}ek \textit{et al} \cite{bulivcek2019uniqueness} for
$r>\frac{3d+2}{d+2}$ and $c_{3},c_{4}>0$ in $\left(  \text{\ref{monotonicity}%
}\right)  $. For ill-posedness results, we refer to the recent paper
\cite{Burczak_et_al2021} where non-uniqueness results were established : for
$r<\frac{2d}{d+2}$ there exists infinitely many Leray--Hopf solutions while
for the case $r\in\left(  \frac{2d}{d+2},\frac{3d+2}{d+2}\right)  $
non-uniqueness of distributional solutions is showed.

\underline{Different forms of the stress tensor :} 
the particular example of other forms of considered stress tensors is the following:
\begin{equation}
\left\{
\begin{array}
[c]{r}%
\mathbb{S}=\left(  \tau_{\ast}+\nu\left\vert \mathbb{D}u\right\vert
^{r-1}\right)  \dfrac{\mathbb{D}u}{\left\vert \mathbb{D}u\right\vert }\text{
if }\left\vert \mathbb{D}u\right\vert >0,\\
\left\vert \mathbb{S}\right\vert \leq\tau_{\ast}\text{ if }\left\vert
\mathbb{D}u\right\vert =0,
\end{array}
\right.  \label{Herschel-Bulkley}%
\end{equation}
where $\tau_{\ast}\nu>0$ are given constants. The additional term $\frac{\D u}{|\D u|}$ corresponds to the viscoplastic effects. Fluids, for which internal
efforts are modeled by $\left(  \text{\ref{Herschel-Bulkley}}\right) $, are
referred to as Herschel--Bulkley fluids, see for instance
\cite{duvant2012inequalities,MalekRajagopal2010,blechta2020classification}. We refer to \cite{malek2005herschel} for an
early result regarding the steady case of the NS equations $\left(
\text{\ref{system}}\right)  $. The viscous stress tensor $\mathbb{S}$ defined
by $\left(  \text{\ref{Herschel-Bulkley}}\right)  $ can be expressed in an implicit way by the following relation:
\[
\nu\left\vert \mathbb{D}u\right\vert ^{r-1}\left(  \mathbb{\tau}^{\ast
}+\left(  \left\vert \mathbb{S}\right\vert -\tau^{\ast}\right)  ^{+}\right)
\frac{\mathbb{D}u}{\left\vert \mathbb{D}u\right\vert }-\left(  \left\vert
\mathbb{S}\right\vert -\tau^{\ast}\right)  ^{+}\mathbb{S}=0,
\]
see \cite{bulicek2012unsteady}. By now, there is an impressive body of work
devoted to implicitly defined stress tensors, see for instance
\cite{eberlein2012existence}, \cite{bulicek2012unsteady},
\cite{bulivcek2023unsteady} and the references cited therein.

Another recent results concern the situation where the stress tensor depends explicitly on time or space, i.e. $\mathbb{S=S}\left(  t,x,\mathbb{D}%
u\right) $. Such models were considered for example in \cite{Gwiazda2010} and \cite{bulivcek2023non}. For other models, involving for example heat-conductivity or inhomogenity, we refer to
\cite{Frehse2010,wroblewska2012unsteady,FeireislLuMalek2015,AbbatielloBulicekKaplicky2022}, to cite just a few.

\paragraph{Compressible case.}
In the case of compressible
fluids much less is known. Recall that the Navier-Stokes system of equations
that govern the motion of compressible barotropic non-Newtonian fluids reads
\begin{equation}
\left\{
\begin{array}
[c]{l}%
\partial_{t}\varrho+\operatorname{div}\left(  \varrho u\right)  =0,\text{ }\\
\partial_{t}\left(  \varrho u\right)  +\operatorname{div}\left(  \varrho
u\otimes u\right)  -\operatorname{div}\mathbb{S}+\nabla p\left(
\varrho\right)  =\rho f+g,
\end{array}%
\begin{array}
[c]{c}%
\text{in }\left(  0,\infty\right)  \times\Omega,\\
\text{in }\left(  0,\infty\right)  \times\Omega,
\end{array}
\right.  \label{NavierStokesCompressible}%
\end{equation}
where $\varrho=\varrho\left(  t,x\right)  \geq0$ represents the density of
fluid, $u=u\left(  t,x\right)  \in\mathbb{R}^{d}$ represents the velocity of
the fluid while $p=p(\varrho)$ is the pressure. The strategy to obtain weak solutions is to perform a limit passage of a regularized equation, but the proof that a suitable limit of approximate solutions is indeed a weak solution for $\left(
\text{\ref{NavierStokesCompressible}}\right)  $ is much more involved than in
the incompressible case. This is due to the fact that the uniform bounds, provided on the sequence of approximate densities by the energy associated with the system, are only in some Lebesgue space.
\textit{Of course, one of the most subtle questions to be answered in the
proof of existence of weak-solutions for }$\left(
\text{\ref{NavierStokesCompressible}}\right)  $ \textit{is how to provide
control for the oscillations of sequence approximating the density}.

Regarding the case of Newtonian isotropic tensors i.e.
\[
\mathbb{S}\left(  \mathbb{D}u\right)  =2\mu\mathbb{D}u+\lambda
\operatorname{div}u\mathbb{I}_{d}%
\]
with $\mu,\lambda>0$ constant, is intimately linked to the mathematical
properties of the effective viscous flux which were exploited by P.L. Lions
\cite{Lions1998} and later by E. Feireisl \cite{Feireisl2001a}. The algebraic
structure of the system drastically changes in the non-Newtonian case or even
in the case of Newtonian anisotropic viscous stress tensors. For the former
case, we refer the reader to the recent results \cite{BreschJabin2018,BreschBurtea2023,BreschBurtea2021} or for simplified models for
compressible fluids \cite{BreschBurtea2020}.

All results that we are aware of regarding compressible non-Newtonian flows
are obtained under some structural assumptions on the viscous
stress tensor that ensure that the divergence of the velocity field belongs or
"almost" belongs to $L^{\infty}$ with respect to the space variable.

The first results were obtained by A.E. Mamontov
\cite{Mamontov1999,Mamontov1999_2}. In his work, the viscous stress tensor is such that
ensures exponentially integrability of the gradient of the vector field. This
is enough to control the propagation of oscillation, using the results for the transport equation obtained by A.V.\ Kazhikhov
and A.E. Mamontov in \cite{Kazhikhov_Mamontov_1998}. Regarding viscoplastic
behavior, we refer to the papers \cite{BasovShelukhin1999}, \cite{Mamontov2007}.

Another interesting result is the one obtained by E. Feireisl, X. Liao and J.
Mal\'{e}k \cite{feireisl2015global} for viscous stress tensors of the form
\[
\mathbb{S}=2\mu_{0}\left(  1+\left\vert \mathbb{D}^{d}u\right\vert
^{2}\right)  ^{\frac{r-2}{2}}\mathbb{D}^{d}u+\frac{b\operatorname{div}%
u}{\left(  1-b^{a}\left\vert \operatorname{div}u\right\vert ^{a}\right)
^{\frac{1}{a}}}\mathbb{I}_{3}%
\]
where $\mu_{0},b,a>0$ and $r\geq\frac{11}{5}$ in three dimensions. The singular
character of $\mathbb{S}$ ensures that the energy naturally associated with
$\left(  \text{\ref{NavierStokesCompressible}}\right)  $ provides the bound
$\operatorname{div}u\in L_{t}^{\infty}L_{x}^{\infty}$. 

Let us also mention the recent paper
\cite{pokorny2022weak}, dealing with
\[
\mathbb{S}=2\left(  \mu_{0}+\mu\left(  \left\vert \mathbb{D}u\right\vert
\right)  \right)  \mathbb{D}u+\lambda\left(  \left\vert \operatorname{div}%
u\right\vert \right)  \operatorname{div}u\mathbb{I}_{d}%
\]
with%
\[
0\leq\mu\left(  z\right)  ,\lambda\left(  z\right)  \leq\frac{C}{z}%
\]
for all $z>0$. The specific form of the viscous stress tensor ensures that one
can recover that the effective viscous flux is bounded in $L_{t}^{\infty}BMO.$ This again is enough to identify the weak and strong limits of the pressure.

We point out that an interesting strategy for providing control for the
oscillations of sequence approximating the density was proposed in the paper
by Zhikov and Pastukhova \cite{ZhikovPastukhova2009}. However, there seems to be a gap
in the authors' argument, already signaled in
\cite{feireisl2015global}, that to our knowledge was not filled. Therefore,
the question of existence for $\left(  \text{\ref{NavierStokesCompressible}%
}\right)  $ for non-Newtonian power-law fluids with $\gamma-$law pressure
i.e.
\[
\mathbb{S}=\mu\left\vert \mathbb{D}u\right\vert ^{r-2}\mathbb{D}%
u+\lambda\left\vert \operatorname{div}u\right\vert ^{r-2}\operatorname{div}%
u\mathbb{I\,}_{d},\text{ }\mu,\lambda>0,\text{ }p\left(  \varrho\right)
=a\varrho^{\gamma}%
\]
is still open.

Results, concerning in particular weak-strong uniqueness, were also obtained in the theory of dissipative solutions
\cite{Abbatiello_et_al2020,Basaric2022} which are solutions verifying $\left(
\text{\ref{NavierStokesCompressible}}\right)  $ up to the divergence of a
positive tensor.

\subsection{Main results}

First, let us specify
what we mean by weak solution for $\left(  \text{\ref{main}}\right)  $.

\begin{definition}
\label{Definition}Consider $d\in\left\{  2,3\right\}  ,$ $r,\gamma>1$ such
that
\begin{equation}
d>r>\frac{3d}{d+2}\text{ and }\gamma>\frac{d\left(  r-1\right)  }%
{(d+2)r-3d}\label{restriction_r_gamma_1}%
\end{equation}
or%
\begin{equation}
d\leq r\text{ and }\gamma>1.\label{restriction_r_gamma_2}%
\end{equation}
A pair $\left(  \varrho,u\right)  \in L^{\frac{r\gamma}{r-1}}\left(
\Omega\right)  \times(W_{0}^{1,r}\left(  \Omega\right)  )^{d}$ is called a 
weak solution to $\left(  \text{\ref{main}}\right)  $, if
\[
\int_{\Omega}
\varrho u\cdot\nabla\varphi\;\mathrm{d}x=0\]
and
\[ \int_{\Omega}
\left( -\varrho u\otimes u+\mathbb{S}(  \mathbb{D}u)\right):\nabla\psi
-p(\varrho)\ddiv\psi \;\mathrm{d}x= \displaystyle\int_{\Omega}
(\varrho f+g)\cdot\psi\;\mathrm{d}x,
\]
for all $\varphi\in\mathcal{C}^{1}\left(  \overline{\Omega}\right)  $ and
$\psi=\left(  \psi_{1},...,\psi_{d}\right)  \in(W_{0}^{1,r}\left(
\Omega\right)  )^{d}.$
\end{definition}

 The
restriction of $\mathcal{C}^{1}(\mathbb{R}^{d})$ functions (respectively of $\mathcal{C}^{\infty}\left(  \mathbb{R}^d\right)  $ functions) is denoted by $\mathcal{C}^{1}\left(  \overline{\Omega}\right)  $ (respectively by $\mathcal{C}^{\infty}\left(  \overline{\Omega}\right)  $).

\begin{remark}
In view of Proposition \ref{prolonged_continuity} from the Appendix B, a proof
of which can be found in \cite{NovotnyStraskraba2004}, see Lemma $3.2.$ page
$158$, distributional solutions for the steady continuity equation
\[
\operatorname{div}\left(  \varrho u\right)  =0\text{ in }\mathcal{D}^{\prime
}\left(  \Omega\right)  ,
\]
for, say $\left(  \varrho,u\right)  \in L^{\frac{r\gamma}{r-1}}\left(
\Omega\right)  \times(W_{0}^{1,r}\left(  \Omega\right)  )^{d}$ have the
property that the functions $\left(  \tilde{\varrho},\tilde{u}\right)  \in
L^{\frac{r}{r-1}}(\mathbb{R}^{d})\times(W_{0}^{1,r}(\mathbb{R}^{d}))^{d}$
obtained by prolonging $\left(  \varrho,u\right)  $ by $0$ in the exterior of
$\Omega$ verify%
\[
\operatorname{div}\left(  \tilde{\varrho}\tilde{u}\right)  =0\text{ in
}\mathcal{D}^{\prime}(\mathbb{R}^{d}).
\]
This is of course equivalent to asking that
\begin{equation}%
{\displaystyle\int_{\Omega}}
\varrho u\cdot\nabla\varphi \;\mathrm{d}x=0, \quad
\text{for all }\varphi\in\mathcal{C}%
^{1}\left(  \overline{\Omega}\right)  .\label{renorm_up_to_boundary}%
\end{equation}

\end{remark}

Our main result concerns the existence of weak solutions to system
$\left(  \text{\ref{main}}\right)  $.

\begin{theorem}
\label{th_st}Consider $d\in\left\{  2,3\right\}$ and $r,\gamma>1$ satisfying (\ref{restriction_r_gamma_1}) or (\ref{restriction_r_gamma_2}).
Let $\mathbb{S}:\mathbb{R}_{\mathrm{sym}}^{d\times d}\rightarrow\mathbb{R}_{\mathrm{sym}}^{d\times d}$ be a
continuous function satisfying the conditions (\ref{growth})-(\ref{monotonicity}) and $p$ given by (\ref{pressure}). Fix $M>0$ and $f,g\in(L^{\infty
}\left(  \Omega\right)  )^{d}$. Then, there exists a weak solution $\left(
\varrho,u\right)  \in L^{\frac{r\gamma}{r-1}}\left(  \Omega\right)
\times(W_{0}^{1,r}\left(  \Omega\right)  )^{d}$ to (\ref{main}), verifying
the following estimates:%
\begin{equation}
\begin{aligned}
\int_\Omega \varrho\;\dd x &= M, \\
\int_{\Omega}\mathbb{S(D}u):\D u \;\mathrm{d}x &\leq C\left(  M,\Omega
,\left\Vert f,g\right\Vert _{L^{\infty}}\right)  ,\\
\left\Vert \varrho\right\Vert _{L^{\frac{r\gamma}{r-1}}} &\leq C\left(
M,\Omega,\left\Vert f,g\right\Vert _{L^{\infty}}\right)  .
\end{aligned}
\label{estimates}%
\end{equation}

\end{theorem}

Moreover, weak solutions constructed above are stable in the following sense:

\begin{tw}\label{Stability}
    Fix $M>0$ and for any $\varepsilon>0$ consider $f_\varepsilon,g_\varepsilon,f,g\in (L^\infty(\Omega))^d$ such that
    \[ Q=\sup_{\varepsilon>0}\|f_\varepsilon,g_\varepsilon\|_{L^\infty}<\infty, \quad (f_\varepsilon,g_\varepsilon) \rightharpoonup (f,g) \quad \text{weakly-$\ast$ in} \quad (L^\infty(\Omega))^d. \]
    Let $(\varrho_\varepsilon,u_\varepsilon)$ be a sequence of weak solutions to (\ref{main}) in the sense of Definition \ref{Definition} with source terms $f_\varepsilon$ and $g_\varepsilon$, satisfying 
    \[ \int_\Omega \varrho_\varepsilon \;\dd x = M \]
    and the uniform bounds
    \begin{equation}\label{uniform} \|\varrho_\varepsilon\|_{L^{\frac{r}{r-1}\gamma}} + \|u_\varepsilon\|_{W^{1,r}} \leq C(M,Q,\Omega), \end{equation}
    Then there exists $(\varrho,u)\in L^{\frac{r\gamma}{r-1}}(\Omega)\times(W_0^{1,r}(\Omega))^d$ such that up to a subsequence (indexed again by $\varepsilon$) we have the strong convergence
    \[\begin{aligned}
        \varrho_\varepsilon \to \varrho \quad &\text{strongly in} \quad L^{\frac{r\gamma}{r-1}}(\Omega), \\
        \D u_\varepsilon \to \D u \quad &\text{strongly in} \quad L^p(\Omega) \quad \text{for any} \quad p<r.
    \end{aligned}\]
    In particular, the limit $(\varrho,u)$ is the weak solution to (\ref{main}) with source terms $f,g$.
\end{tw}

Our last main result deals with existence of weak solutions for
Herschel--Bulkley fluids. More precisely, with a slight abuse of notations
regarding $\mathbb{S}$, we obtain that

\begin{theorem}
\label{Viscoplastic_is_fantastic}Consider $d\in\left\{  2,3\right\}  ,$
$r,\gamma>1$ verifying $\left(  \text{\ref{restriction_r_gamma_1}}\right)  $
or $\left(  \text{\ref{restriction_r_gamma_2}}\right)  $. In addition, assume that $\gamma\in (1,2]$. Fix $\alpha>0$, $\beta\geq 0$, 
$f,g\in(L^{\infty}\left(  \Omega\right)  )^{d}$ and a nonnegative function $\check{\varrho}\in L^\infty(\Omega)$. Then, there exists a weak solution
$\left(  \varrho,u,\mathbb{S}\right)  \in L^{\frac{r\gamma}{r-1}}\left(
\Omega\right)  \times(W_{0}^{1,r}\left(  \Omega\right)  )^{d}\times
(L^{\frac{r}{r-1}}\left(  \Omega\right)  )^{d\times d}$ to the system%
\begin{equation}
\left\{
\begin{array}
[c]{ll}%
\alpha\varrho+\operatorname{div}\left(  \varrho u\right)  =\alpha\check\varrho \text{ } & \text{in }\Omega,\\
\beta\varrho u+\operatorname{div}\left(  \varrho u\otimes u\right)  -\operatorname{div}%
\mathbb{S}+\nabla\varrho^{\gamma}=\rho f+g & \text{in }\Omega,\\
\mathbb{S}=\left(  \tau_{\ast}+\nu\left\vert \mathbb{D}u\right\vert
^{r-1}\right)  \dfrac{\mathbb{D}u}{\left\vert \mathbb{D}u\right\vert } &
\text{in }\left\{  \left\vert \mathbb{D}u\right\vert >0\right\}  ,\\
\left\vert \mathbb{S}\right\vert \leq\tau_{\ast} & \text{in }\left\{
\left\vert \mathbb{D}u\right\vert =0\right\}  ,
\end{array}
\right.  \label{visco-plastic}%
\end{equation}
verifying the following estimates:%
\begin{equation}
\left\Vert \varrho\right\Vert _{L^{\frac{r\gamma}{r-1}}}+\left\Vert \nabla
u\right\Vert _{L^{r}}+\left\Vert \mathbb{S}\right\Vert _{L^{\frac{r}{r-1}}%
}\leq C\left(  M,\Omega,\left\Vert f,g\right\Vert _{L^{\infty}}\right)  .
\label{estimate_viscoplastic}%
\end{equation}

\end{theorem}

System $\left(  \text{\ref{visco-plastic}}\right)  $ can be viewed as an
implicit time-discretization of the non-steady case and, besides being
interesting in itself, it can represent a first step towards the construction
of solutions in the non-steady case. Unfortunately, the
case $\alpha=0$ in $\left(  \text{\ref{visco-plastic}}\right)  $ seems to be out of reach within our method.

\subsection{Main ideas of the proof}

Weak solutions to (\ref{main}) are obtained as the limit of solutions $\left(  \varrho
_{\varepsilon},u_{\varepsilon}\right)  $ to a regularized variant of $\left(
\text{\ref{main}}\right)  $, verifying
\[
\left\Vert \varrho_{\varepsilon}\right\Vert _{L^{\frac{r\gamma}{r-1}}%
}+\left\Vert \nabla u_{\varepsilon}\right\Vert _{L^{r}}\leq C
\]
uniformly with respect to $\varepsilon$.
As it is customary, we will denote by $\overline{\left(  \cdot\right)  }$ weak
limits of functions of $\varrho_{\varepsilon}$ or $\nabla u_{\varepsilon}$. By
straightforward compactness techniques one can establish that%
\[
\left\{
\begin{array}
[c]{l}%
\operatorname{div}\left(  \varrho u\right)  =0,\text{ }\\
\operatorname{div}\left(  \varrho u\otimes u\right)  -\operatorname{div}%
\overline{\mathbb{S}\left(  \mathbb{D}u\right)  }+\nabla\overline
{\varrho^{\gamma}}=\rho f+g,
\end{array}%
\begin{array}
[c]{r}%
\text{in }\Omega,\\
\text{in }\Omega,
\end{array}
\right.
\]
\textit{Of course, one of the most subtle points is to justify that }%
\[
\overline{\mathbb{S}\left(  \mathbb{D}u\right)  }=\mathbb{S}\left(
\mathbb{D}u\right)  \text{ and }\overline{\varrho^{\gamma}}=\varrho^{\gamma
}\text{ a.e. in }\Omega.
\]
\textit{This is achieved by combining the monotonicity technique with the
theory of renormalized solutions for the steady continuity equation.}

The key ingredient is to obtain a certain identity, which will allow comparison between weak and strong limits. First,
one writes the equation for the local energy associated with the system (it is done by testing the momentum equation by $\varphi
u_{\varepsilon}$ with $\varphi\in\mathcal{C}^{\infty}\left(  \mathbb{R}%
^{d}\right)  $) and passes to the limit. Then, we subtract the
local energy associated to the limit system. From this we obtain%
\[
\frac{1}{\gamma-1}\operatorname{div}\left(  \overline{\varrho^{\gamma}%
}u\right)  +\overline{\varrho^{\gamma}}\operatorname{div}u+\overline
{\mathbb{S}\left(  \mathbb{D}u\right)  :\mathbb{D}u}-\overline{\mathbb{S}%
\left(  \mathbb{D}u\right)  }:\mathbb{D}u=0,
\]
see Proposition \ref{Prop_identity} for a more precise statement.
"Multiplying" the above equation by $\frac{1}{\gamma}\overline{\varrho
^{\gamma}}^{\frac{1}{\gamma}-1}$ we end up with%
\begin{equation}
\frac{1}{\gamma-1}\operatorname{div}(\overline{\varrho^{\gamma}}^{\frac
{1}{\gamma}}u)+\left(  \overline{\mathbb{S}\left(  \mathbb{D}u\right)
:\mathbb{D}u}-\overline{\mathbb{S}\left(  \mathbb{D}u\right)  }:\mathbb{D}%
u\right)  \frac{1}{\gamma}\overline{\varrho^{\gamma}}^{\frac{1}{\gamma}%
-1}=0.\label{identity_control_defect_mesure}%
\end{equation}
By integration, one obtains%
\[
\int_{\Omega}\left(  \overline{\mathbb{S}\left(  \mathbb{D}u\right)
:\mathbb{D}u}-\overline{\mathbb{S}\left(  \mathbb{D}u\right)  }:\mathbb{D}%
u\right)  \overline{\varrho^{\gamma}}^{\frac{1}{\gamma}-1}\mathrm{d}x=0,
\]
and owing to the monotonicity we end up with
\[
\overline{\mathbb{S}\left(  \mathbb{D}u\right)  :\mathbb{D}u}=\overline
{\mathbb{S}\left(  \mathbb{D}u\right)  }:\mathbb{D}u,
\]
which is sufficient in order to obtain strong convergence (up to a
subsequence) for $\nabla u_{\varepsilon}$ towards $\nabla u$. Then, one hopes
to find some ad-hoc argument in order to recover the strong convergence of the
density. This strategy was used in \cite{BreschBurtea2021} for anisotropic
Newtonian viscous stress tensors (linear dependence w.r.t. $\mathbb{D}u$)
although a special form of the latter was required in order to carry out the proof.

Of course, the reader may have observed that there are a few informal
manipulations in order to arrive at $\left(
\text{\ref{identity_control_defect_mesure}}\right)  $. First of all, since
$\gamma>1$, the term $\frac{1}{\gamma}\left(  \overline{\varrho^{\gamma}%
}\right)  ^{\frac{1}{\gamma}-1}$ is singular on the set
$\left\{  \overline{\varrho^{\gamma}}=0\right\}.$ However,
one can bypass this problem by multiplying with $\frac
{1}{\gamma}\left(  \overline{\varrho^{\gamma}}+h\right)
^{\frac{1}{\gamma}-1}$ for positive $h>0$ and afterwards take the limit when
$h\to 0.$ Perhaps an even more serious obstacle towards the rigorous
justification of $\left(  \text{\ref{identity_control_defect_mesure}}\right)
$ seems to be that $\overline{\mathbb{S}\left(  \mathbb{D}u\right)  :\D
u}-\mathbb{S}\left(  \mathbb{D}u\right)  :\D u$ is a measure and the
product with $\left(  \overline{\varrho^{\gamma}}\right)
^{\frac{1}{\gamma}-1}$ does not make sense, the latter being only in some Lebesgue space. 

In \cite{BreschBurtea2021} the authors dealt with this obstacle by proving higher integrability estimates for $\nabla u$. This was possible, assuming the fluid was "almost isotropic". This unfortunately seems to be out of reach in the framework of this paper. \textit{Instead, we apply a certain measure-theoretical argument, based on Egorov's theorem. It allows to perform all necessary estimates on a slightly smaller set $\Omega_\eta$, on which $(\overline{\varrho^\gamma}+h)^{1/\gamma-1}$ is already continuous.} The core of the argument is the following (see Proposition \ref{propo_limit}):
\begin{prop*}
    Consider $\gamma>1,$ $u\in(W_{0}^{1,r}\left(
\Omega\right)  )^{d}$, $b\in L^{\frac{r}{r-1}}\left(  \Omega\right)$ nonnegative and a positive measure $\mathcal{R}\in\mathcal{M}\left(  \Omega\right)  $ such that the identity
\[
\frac{-1}{\gamma-1}\int_{\Omega}b u\nabla\varphi\text{ }\mathrm{d}%
x+\int_{\Omega}b\mathrm{div\,}u\cdot\varphi\text{ }\mathrm{d}%
x+\int_{\Omega}\varphi\text{ }\mathrm{d}\mathcal{R}=0\text{,}
\]
holds for all $\varphi\in\mathcal{C}^{1}(\overline{\Omega})$. Then, for all
$\eta>0$ there exists a compact set $\Omega_{\eta}\subset\Omega$ such that
$\left\vert \Omega\backslash\Omega_{\eta}\right\vert \leq\eta$ and $\mathcal{R}\left(
\Omega_{\eta}\right)  =0$.
\end{prop*}
By doing so we are able to recover that $\nabla u_\varepsilon\to\nabla u$ strongly in $L^q(\Omega)$ for all $q<r$. Then, to recover convergence of the density we
use the momentum equation to see that the pressure verifies the Poisson
equation with a strongly convergent term on the right hand side. In particular,
Theorem \ref{th_st} generalizes the main result in \cite{BreschBurtea2021} in
the sense that we do not require anymore the "almost isotropic" condition.

The (informal) identity $\left(  \text{\ref{identity_control_defect_mesure}%
}\right)  $ seems to have been first noticed by Zhikov and Pastukhova in
\cite{ZhikovPastukhova2009} for the non-steady case where one has instead%
\[
\frac{\gamma}{\gamma-1}\left(  \partial_{t}\left(  \overline{\varrho^{\gamma}%
}-\varrho^{\gamma}\right)  ^{\frac{1}{\gamma}}+\operatorname{div}\left(
\left(  \overline{\varrho^{\gamma}}-\varrho^{\gamma}\right)  ^{\frac{1}%
{\gamma}}u\right)  \right)  +\left(  \overline{\mathbb{S}\left(
\mathbb{D}u\right)  :\nabla u}-\mathbb{S}\left(  \mathbb{D}u\right)  :\nabla
u\right)\left(  \overline{\varrho^{\gamma}}-\varrho^{\gamma
}\right)  ^{\frac{1}{\gamma}-1}=0.
\]
However, as it was mentioned in the introduction, there seems to be a gap in their proof. This problem was also signaled in \cite{feireisl2015global} and up to our knowledge
has not been filled. This comes from the fact that in the evolution case one
does not have any information on the time derivative for the velocity $u$ and
therefore in the energy equation the term in the energy equation which comes
from multiplying the viscous stress tensor with the velocity
$\operatorname{div}\left(  \mathbb{S}\left(  \mathbb{D}u\right)  u\right)  $
raises serious problems. Independently, one of the
authors of this paper and D. Bresch rediscovered this formula later in
\cite{BreschBurtea2020} and used it to prove existence for the anisotropic
quasi-stationary case (transport+Stokes). We also cite the recent result
\cite{BreschBurtea2023} where the authors prove existence of
intermediate-regularity weak solutions.

The rest of this paper unfolds as follows. Section
\ref{Proof of 1st main Theorem} is devoted to the proof of Theorem
\ref{th_st}. The proof is divided into three main parts. In a first step,
which makes the object of Section \ref{Approximate problems}, we introduce the
approximate system. It can be solved by standard nonlinear analysis
techniques, which are presented in the Appendix. In Section \ref{Uniform_Estimates}, we provide uniform estimates
for solutions of the (reduced) approximate systems. The more subtle stability
arguments are presented in Section \ref{Limit passage with epsilon}. This
section ends with the arguments necessary in order to conclude the proof of
Theorem \ref{Stability}. In Section \ref{Proof of Theorem Viscoplastic} we prove Theorem
\ref{Viscoplastic_is_fantastic}. Appendix A contains the proof of existence
for solutions for the approximate system. In Appendix B, we gather non-linear
analysis tools that we use all along the paper. All but one of the results
gathered in this appendix are borrowed from \cite{NovotnyStraskraba2004}.

\section{Proofs of Theorem \ref{th_st} and Theorem \ref{Stability}%
\label{Proof of 1st main Theorem}}

\subsection{Approximate problems\label{Approximate problems}}

A weak solution for system $\left(  \text{\ref{main}}\right)  $ is obtained as
the limit of solutions of the following system when the regularization
parameters $\alpha,\delta,\varepsilon,\eta\in\left(  0,1\right)  $ tend to
$0$:%
\begin{equation}
\left\{
\begin{array}
[c]{l}%
-\varepsilon\Delta\varrho+\eta\left(  \varrho-\dfrac{M}{\left\vert
\Omega\right\vert }\right)  +\operatorname{div}\left(  \varrho\omega_{\delta
}\ast T_{\delta}\left(  u\right)  \right)  =0,\\
\operatorname{div}\left(  \varrho\left(  \omega_{\delta}\ast T_{\delta
}(u)\right)  \otimes u\right)  -\operatorname{div}\mathbb{S}(\D u)-\alpha
\operatorname{div}\left(  \left\vert \nabla u\right\vert ^{q-2}\nabla
u\right)  +T_{\delta}(\omega_{\delta}\ast\widetilde{\nabla\varrho^{\gamma}})\\
=\frac{\eta}{2}\varrho u-\varepsilon\nabla u\nabla\varrho+\varrho f+g,\\
\varrho\geq0,\text{ }%
{\displaystyle\int_{\Omega}}
\varrho\mathrm{d}x=M,\\
\frac{\partial\varrho}{\partial\overrightarrow{n}}=0\text{ },u=0\text{ on
}\partial\Omega
\end{array}
\right.  \label{Approximare_eps_delta_eta_alpha_main}%
\end{equation}
Above,
\[
\omega_{\delta}\left(  \cdot\right)  =\frac{1}{\delta^{3}}\omega\left(
\frac{1}{\delta}\cdot\right)  ,\text{ }T_{\delta}(v)\left(  x\right)
=\boldsymbol{1}_{\Omega_{2\delta}}\left(  x\right)  v\left(  x\right)  ,
\]
where $\omega\in\mathcal{D(}\mathbb{R}^{d})$ is a smooth, non-negative, radial
function which is compactly supported in the unit ball centered at the origin
and with integral $1$, while%
\[
\Omega_{2\delta}=\left\{  x\in\Omega:d\left(  x,\partial\Omega\right)
\geq2\delta\right\}  .
\]
By the $\widetilde{\cdot}$ we denote the trivial extension on $\mathbb{R}^{d}$
by $0$. Of course, we have that for any $u\in\left(  W^{1,q}\left(
\Omega\right)  \right)  ^{d},$ $\omega_{\delta}\ast T_{\delta}\left(
u\right)  \in W_{0}^{1,q}\left(  \Omega\right)  $.

The existence of solutions to the above system is a consequence of the
Leray-Schauder fixed point and the estimates in Sobolev spaces. More
precisely, we have that

\begin{theorem}
\label{Existence_4_levels_main}Consider $\alpha,\delta,\varepsilon,\eta
\in\left(  0,1\right)  $, $d\in\left\{  2,3\right\}  ,$ $\gamma>1$,
$q>d$ and $r>\frac{d}{2}$. Let $\mathbb{S}:\mathbb{R}%
_{sym}^{d\times d}\rightarrow\mathbb{R}_{sym}^{d\times d}$ be a continuous
function verifying the conditions (\ref{growth})-(\ref{monotonicity}). For all $M>0$ and $f,g\in(L^{\infty}\left(  \Omega\right)  )^{d}%
$, there exists a solution $\left(  \varrho,u\right)  \in W^{2,2}\left(
\Omega\right)  \times(W_{0}^{1,q}\left(  \Omega\right)  )^{d}$ to $\left(
\text{\ref{Approximare_eps_delta_eta_alpha_main}}\right)  $ verifying the
identities
\begin{equation}\label{mass_conservation}
    \int_\Omega\varrho\;\dd x = M,
\end{equation}
and
\begin{align}
\int_{\Omega}\mathbb{S}(\D u) &  :\D u\text{ }\mathrm{d}x+\frac{4\varepsilon
}{\gamma}%
{\displaystyle\int_{\Omega}}
\left\vert \nabla\varrho^{\frac{\gamma}{2}}\right\vert ^{2}\mathrm{d}%
x+\alpha\int_{\Omega}\left\vert \nabla u\right\vert ^{q}\mathrm{d}%
x+\frac{\gamma\eta}{2\left(  \gamma-1\right)  }%
{\displaystyle\int_{\Omega}}
\varrho^{\gamma}\mathrm{d}x+\frac{\eta M}{\left\vert \Omega\right\vert }%
\int_{\Omega}\frac{\left\vert u\right\vert ^{2}}{2}\mathrm{d}x\nonumber\\
&  =\frac{\gamma\eta}{\left(  \gamma-1\right)  }\frac{M}{\left\vert
\Omega\right\vert }%
{\displaystyle\int_{\Omega}}
\varrho^{\gamma-1}\mathrm{d}x+\int_{\Omega}\left(  \varrho f+g\right)  u\text{
}\mathrm{d}x,\label{identity_1}%
\end{align}
along with the following estimates:%
\begin{equation}
\begin{aligned}
\int_{\Omega}\mathbb{S}(\D u):\D u\;\mathrm{d}x+ & \frac{4\varepsilon
}{\gamma}\int_{\Omega}
\left\vert \nabla\varrho^{\frac{\gamma}{2}}\right\vert ^{2}\mathrm{d}x 
+ \alpha\int_{\Omega}\left\vert \nabla u\right\vert ^{q}\mathrm{d}%
x\\
&+\frac{\gamma\eta}{2\left(  \gamma-1\right)  }\int_{\Omega}
\varrho^{\gamma}\;\mathrm{d}x+\frac{\eta M}{\left\vert \Omega\right\vert }%
\int_{\Omega}\frac{\left\vert u\right\vert ^{2}}{2}\;\mathrm{d}x\leq C\left(
M,\Omega,\varepsilon,\left\Vert f,g\right\Vert _{L^{\infty}}\right)
\end{aligned}
  \label{Estim_exist_4_levels}%
\end{equation}
and
\begin{equation}\label{est_eps1}
    \varepsilon\|\nabla^2\varrho\|_{L^r}\leq C(M,\Omega,\varepsilon,\|f,g\|_{L^\infty}).
\end{equation}

\end{theorem}

\begin{remark}
\bigskip The restriction $r>\frac{d}{2}$ is needed in order to justify
$\left(  \text{\ref{Estim_exist_4_levels}}\right)  $. In two or three dimensions, it follows however from the assumptions (\ref{restriction_r_gamma_1})-(\ref{restriction_r_gamma_2}).
\end{remark}

Let us point out that the approximate system with the mollified truncated
velocity $\omega_{\delta}\ast T_{\delta}\left(  u\right)  $ that we consider
here is somehow non-standard and that a similar system was proposed in
\cite{BreschBurtea2020}. The proof of Theorem \ref{Existence_4_levels_main} is
by now standard in the literature and follows the main ideas that can be found
in chapter $4$ of the classical reference \cite{NovotnyStraskraba2004}. For
the sake of completeness, we included a proof of Theorem
\ref{Existence_4_levels_main} in Appendix A.

\subsection{Uniform estimates\label{Uniform_Estimates}}

Before announcing the main result of this section, let us remind the readers that our results
will be obtained in $d\in\left\{  2,3\right\}  $ space dimensions. We consider
$r,\gamma>1$ such that%
\begin{equation}
d>r>\frac{3d}{d+2}\text{ and }\gamma>\frac{d\left(  r-1\right)  }%
{(d+2)r-3d}\label{cond_1}%
\end{equation}
or%
\begin{equation}
d\leq r\text{ and }\gamma>1.\label{cond_2}%
\end{equation}
Furthermore, we introduce $q_{1}^{\ast},q_{2}^{\ast}\in(1,\frac{r\gamma}%
{r-1}]$ given by%
\begin{equation}
\frac{1}{q_{1}^{\ast}}=\left\{
\begin{array}
[c]{r}%
\frac{1}{r}-\frac{1}{d}+\frac{r-1}{r\gamma}\text{ when }r<d,\\
\text{any }q\text{ : }\frac{1}{q}>\frac{r-1}{r\gamma}\text{when }r=d,\\
\frac{r-1}{r\gamma}\text{ when }r>d.
\end{array}
\right.  \label{definitie_q1_ast}%
\end{equation}
respectively%
\begin{equation}
\frac{1}{q_{2}^{\ast}}=\left\{
\begin{array}
[c]{r}%
\frac{2}{r}-\frac{2}{d}+\frac{r-1}{r\gamma}\text{ when }r<d,\\
\text{any }q\text{ : }\frac{1}{q}>\frac{r-1}{r\gamma}\text{when }r=d,\\
\frac{r-1}{r\gamma}\text{ when }r>d.
\end{array}
\right.  \label{definitie_q2_ast}%
\end{equation}

As a consequence of Theorem \ref{Existence_4_levels_main} we obtain the following

\begin{theorem}
\label{Existence_1_level}Consider $\varepsilon\in\left(  0,1\right)  $,
$d\in\left\{  2,3\right\}  $ and $r,\gamma>1$ verifying $\left(
\text{\ref{cond_1}}\right)  $ or $\left(  \text{\ref{cond_2}}\right)  $. Let
$\mathbb{S}:\mathbb{R}_{sym}^{d\times d}\rightarrow\mathbb{R}_{sym}^{d\times
d}$ be a continuous function verifying the conditions (\ref{growth})-(\ref{monotonicity}). For all $M>0$ and $f,g\in(L^{\infty
}\left(  \Omega\right)  )^{d}$, there exists a solution $\left(
\varrho,u\right)  \in W^{2,r}\left(  \Omega\right)  \times(W_{0}^{1,r}\left(
\Omega\right)  )^{d}$ for the system\footnote{$(\nabla u\nabla\varrho)_i = \sum_{j=1}^d \partial_{x_j}u_i\partial_{x_j}\varrho$}
\begin{equation}
\left\{
\begin{array}
[c]{l}%
\operatorname{div}\left(  \varrho u\right)  =\varepsilon\Delta\varrho,\\
\operatorname{div}\left(  \varrho u\otimes u\right)  -\operatorname{div}%
\mathbb{S}\left(  \mathbb{D}u\right)  +\nabla\varrho^{\gamma}=-\varepsilon
\nabla u\nabla\varrho+\varrho f+g,\\
\varrho\geq0,\text{ }%
{\displaystyle\int_{\Omega}}
\varrho\text{ }\mathrm{d}x=M,\text{ }\dfrac{\partial\rho}{\partial
\overrightarrow{n}}=0\text{ on }\partial\Omega\\
u=0\text{ on }\partial\Omega
\end{array}
\right.  \label{system_eps}%
\end{equation}
verifying the following estimates:%
\begin{equation}\label{est_eps}
    \varepsilon\|\nabla^2\varrho_\varepsilon\|_{L^r}\leq C(M,\Omega,\varepsilon,\|f,g\|_{L^\infty})
\end{equation}
and
\begin{equation}
\left\{
\begin{array}
[c]{l}%
{\displaystyle\int_{\Omega}}
\mathbb{S}\left(  \mathbb{D}u\right)  :\mathbb{D}u\text{ }\mathrm{d}%
x+\dfrac{4\varepsilon}{\gamma}%
{\displaystyle\int_{\Omega}}
\left\vert \nabla\varrho^{\frac{\gamma}{2}}\right\vert ^{2}\mathrm{d}x\leq
C\left(  M,\Omega,\left\Vert f,g\right\Vert _{L^{\infty}}\right)  ,\\
\left\Vert \varrho\right\Vert _{L^{\frac{r\gamma}{r-1}}}\leq C\left(
M,\Omega,\left\Vert f,g\right\Vert _{L^{\infty}}\right)  .
\end{array}
\right.  \label{estimations_niveau_eps}%
\end{equation}
Moreover, for any $\eta_{1},\eta_{2},\eta_{3}>0$ small enough there exist
$\alpha_i(\eta_i)\in (0,1)$, $i=1,2,3$, such that%
\begin{equation}
\varepsilon\left\Vert \nabla\varrho\right\Vert _{L^{q_{1}^{\ast}-\eta_1}%
}+\varepsilon\left\Vert \nabla\varrho\nabla u\right\Vert _{L^{\frac
{rq_{1}^{\ast}}{r+q_{1}^{\ast}}-\eta_{2}}}+\varepsilon\left\Vert \nabla
\varrho\otimes u\right\Vert _{L^{q_{2}^{\ast}-\eta_{3}}}\leq C\left(
M,\Omega,\left\Vert f,g\right\Vert _{L^{\infty}}\right)  \left(
\varepsilon^{\alpha\left(  \eta_{1}\right)  }+\varepsilon^{\alpha_{2}\left(
\eta_{2}\right)  }+\varepsilon^{\alpha_{3}\left(  \eta_{3}\right)  }\right)
.\label{estimation_passage_limite_eps}%
\end{equation}

\end{theorem}

\begin{proof}
    The existence follows by
considering a sequence of solutions of $\left(
\text{\ref{Approximare_eps_delta_eta_alpha_main}}\right)  $ and passing to the
limit with $\alpha,\delta,\eta\rightarrow0$. Owing to the fact that the
estimates $\left(  \text{\ref{Estim_exist_4_levels}}\right)  $ are verified
uniformly w.r.t. $\alpha,\delta,\eta$ we end up with a limiting couple
$\left(  \varrho,u\right)  \in W^{2,r}\left(  \Omega\right)  \cap(W_{0}%
^{1,r}(\Omega))^{d}$. The fact that $\left(  \varrho,u\right)  $ is a solution
for system $\left(  \text{\ref{system_eps}}\right)  $ is proved by using
classical monotonicity compactness arguments and the Rellich--Kondrachov
theorem we therefore leave it as an exercise for the reader. The estimate (\ref{est_eps}) is
inherited from the lower semicontinuity and the fact that solutions of
$\left(  \text{\ref{Approximare_eps_delta_eta_alpha_main}}\right)  $ verify estimates (\ref{Estim_exist_4_levels})-
(\ref{est_eps1}). Note that at this point the constant in (\ref{Estim_exist_4_levels}) depends also on $\varepsilon$. The bounds (\ref{estimations_niveau_eps}) are shown in the following Proposition:
\begin{proposition}
    Under the assumptions of Theorem \ref{Existence_1_level}, $(\varrho_\varepsilon,u_\varepsilon)$ satisfy the estimates (\ref{estimations_niveau_eps}).
\end{proposition}
\begin{proof} Observe
that from $\left(  \text{\ref{identity_1}}\right)  $ we have that%
\begin{equation}
\int_{\Omega}\mathbb{S}(\mathbb{D}u):\mathbb{D}u\text{ }\mathrm{d}x+\frac{4\varepsilon
}{\gamma}%
{\displaystyle\int_{\Omega}}
\left\vert \nabla\varrho^{\frac{\gamma}{2}}\right\vert ^{2}\mathrm{d}x\leq
\int_{\Omega}(\varrho f+g)u\text{ }\mathrm{d}x.\label{estimation_energie_eps}%
\end{equation}

\underline{The case $r<d.$} From $\left(  \text{\ref{estimation_energie_eps}%
}\right)  $, $\left(  \text{\ref{growth}}\right)  $ and Korn's
inequality, it follows that%
\[
\left\Vert \nabla u\right\Vert _{L^{r}}^{r}\leq C\left\Vert \varrho\right\Vert
_{L^{\frac{1}{1+\frac{1}{d}-\frac{1}{r}}}}\left\Vert f\right\Vert _{L^{\infty
}}\left\Vert u\right\Vert _{L^{\frac{dr}{d-r}}}+C\left(  \Omega\right)
\left\Vert g\right\Vert _{L^{\infty}}\left\Vert u\right\Vert _{L^{\frac
{dr}{d-r}}}.
\]
Consequently,%
\begin{equation}
\left\Vert \nabla u\right\Vert _{L^{r}}\leq C\left(  \Omega,\left\Vert
f,g\right\Vert _{L^{\infty}}\right)  (\left\Vert \varrho\right\Vert
_{L^{\frac{1}{1+\frac{1}{d}-\frac{1}{r}}}}^{\frac{1}{r-1}}+1).\label{estim_nabla_u}%
\end{equation}
We precise that the condition $\left(  \text{\ref{cond_1}}\right)  $ imposed
on $r$ and $\gamma$ ensures that that the integrability of $\rho\left\vert
u\right\vert ^{2}$ is higher than that of $\mathbb{S}$:%
\begin{align*}
\frac{r-1}{r\gamma}+\frac{2\left(  d-r\right)  }{dr}%
<\frac{r-1}{r} & \Leftrightarrow\frac{2\left(  d-r\right)  }{dr}<\frac
{r-1}{r}\frac{\gamma-1}{\gamma},\\
&  \Leftrightarrow\frac{2\left(  d-r\right)  }{d\left(  r-1\right)
}<\frac{\gamma-1}{\gamma},\\
&  \Leftrightarrow\frac{1}{\gamma}<1-\frac{2\left(  d-r\right)  }{d\left(
r-1\right)  }=\frac{(d+2)r-3d}{d\left(  r-1\right)  },\\
&  \Leftrightarrow r>\frac{3d}{d+2}\text{ and }\gamma>\frac{d\left(
r-1\right)  }{(d+2)r-3d}.
\end{align*}

From $\left(  \text{\ref{estim_nabla_u}}\right)  $ we infer that%
\begin{align}
\left\Vert \varrho u\otimes u\right\Vert _{L^{\frac{r}{r-1}}}  &  \leq
C\left(  \Omega\right)  \left\Vert \varrho\right\Vert _{L^{\frac{1}{1+\frac
{2}{d}-\frac{3}{r}}}}\left\Vert \nabla u\right\Vert _{L^{r}}^{2}\nonumber\\
&  \leq C\left(  \Omega,\left\Vert f,g\right\Vert _{L^{\infty}}\right)
\left\Vert \varrho\right\Vert _{L^{\frac{1}{1+\frac{2}{d}-\frac{3}{r}}}%
}(\left\Vert \varrho\right\Vert _{L^{\frac{1}{1+\frac{1}{d}-\frac{1}{r}}}%
}^{\frac{2}{r-1}}+1). \label{estimare_rho_u^2_3}%
\end{align}

Next, using Proposition \ref{caldura} we obtain that
\begin{align}
\varepsilon\left\Vert \nabla u\nabla\varrho\right\Vert _{L^{\frac{1}%
{1+\frac{1}{d}-\frac{1}{r}}}}  &  \leq\varepsilon\left\Vert \nabla
\varrho\right\Vert _{L^{\frac{1}{1+\frac{1}{d}-\frac{2}{r}}}}\left\Vert \nabla
u\right\Vert _{L^{r}}\leq C\left(  \Omega\right)  \left\Vert \rho u\right\Vert
_{L^{\frac{1}{1+\frac{1}{d}-\frac{2}{r}}}}\left\Vert \nabla u\right\Vert
_{L^{r}}\nonumber\\
&  \leq C\left(  \Omega\right)  \left\Vert \varrho\right\Vert _{L^{\frac
{1}{1+\frac{2}{d}-\frac{3}{r}}}}\left\Vert u\right\Vert _{L^{\frac{dr}{d-r}}%
}\left\Vert \nabla u\right\Vert _{L^{r}}\nonumber\\
&  \leq C\left(  \Omega,\left\Vert f,g\right\Vert _{L^{\infty}}\right)
\left\Vert \varrho\right\Vert _{L^{\frac{1}{1+\frac{2}{d}-\frac{3}{r}}}%
}(\left\Vert \varrho\right\Vert _{L^{\frac{1}{1+\frac{1}{d}-\frac{1}{r}}}%
}^{\frac{2}{r-1}}+1). \label{eps_nabla_u_nabla_rho}%
\end{align}

Let us remark that using a density argument, we can test the momentum equation
with test functions belonging to $W_{0}^{1,r}\left(  \Omega\right)  .$ Let
$\Psi=\mathcal{B}\left(  \varrho^{\frac{\gamma}{r-1}}-\frac{1}{|\Omega|}%
\int_{\Omega}\varrho^{\frac{\gamma}{r-1}}\mathrm{d}x\right)  $, where
$\mathcal{B}$ is the Bogovskii operator see Proposition \ref{Bogo} from
Appendix B. We have that%
\begin{equation}
\left\Vert \Psi\right\Vert _{L^{\frac{dr}{d-r}}}\leq C\left(  \Omega\right)
\left\Vert \nabla\Psi\right\Vert _{L^{r}}\leq C\left(  M,\Omega\right)
\left\Vert \varrho\right\Vert _{L^{\frac{r\gamma}{r-1}}}^{\frac{\gamma}{r-1}%
}.\label{estim_psi}%
\end{equation}
Since $r>\frac{3d}{d+2}\geq\frac{d}{2}$ we have $\varrho^{\gamma}\in
L^{\frac{d}{d-2}}\subseteq L^{\frac{r}{r-1}}$. We use $\Psi$ as a test
function in the momentum equation and we get that%
\begin{align}
\Vert\varrho\Vert_{L^{\frac{r\gamma}{r-1}}}^{\frac{r\gamma}{r-1}} &  \leq
\frac{1}{|\Omega|}\int_{\Omega}\varrho^{\frac{\gamma}{r-1}}\mathrm{d}%
x\cdot\Vert\varrho\Vert_{L^{\gamma}}^{\gamma}\newline+\Vert\varrho u\otimes
u\Vert_{L^{\frac{r}{r-1}}}\Vert\nabla\Psi\Vert_{L^{r}}+\Vert\mathbb{S}%
\Vert_{L^{\frac{r}{r-1}}}\Vert\nabla\Psi\Vert_{L^{r}}\nonumber\\
&  +\varepsilon\left\Vert \nabla u\nabla\varrho\right\Vert _{L^{\frac
{1}{1+\frac{1}{d}-\frac{1}{r}}}}\left\Vert \Psi\right\Vert _{L^{\frac{dr}%
{d-r}}}+\left\Vert f\right\Vert _{L^{\infty}}\left\Vert \varrho\right\Vert
_{L^{\frac{1}{1+\frac{1}{d}-\frac{1}{r}}}}\left\Vert \Psi\right\Vert
_{L^{\frac{dr}{d-r}}}+C\left(  \Omega\right)  \left\Vert g\right\Vert
_{L^{\infty}}\left\Vert \Psi\right\Vert _{L^{\frac{dr}{d-r}}}%
.\label{estimare_rho_cu_Bogo}%
\end{align}

Now using $\left(  \text{\ref{estim_psi}}\right)  $ along with $\left(
\text{\ref{estimare_rho_u^2_3}}\right)  $ and $\left(
\text{\ref{eps_nabla_u_nabla_rho}}\right)  $ we get that%
\begin{align}
&  \Vert\varrho u\otimes u\Vert_{L^{\frac{r}{r-1}}}\Vert\nabla\Psi\Vert
_{L^{r}}+\varepsilon\left\Vert \nabla u\nabla\varrho\right\Vert _{L^{\frac
{1}{1+\frac{1}{d}-\frac{1}{r}}}}\left\Vert \Psi\right\Vert _{L^{\frac{dr}%
{d-r}}}\nonumber\\
&  \leq C\left(  \Omega,\left\Vert f,g\right\Vert _{L^{\infty}}\right)
\left\Vert \varrho\right\Vert _{L^{\frac{1}{1+\frac{2}{d}-\frac{3}{r}}}%
}(\left\Vert \varrho\right\Vert _{L^{\frac{1}{1+\frac{1}{d}-\frac{1}{r}}}%
}^{\frac{2}{r-1}}+1)\left\Vert \varrho\right\Vert _{L^{\frac{r\gamma}{r-1}}%
}^{\frac{\gamma}{r-1}}. \label{rhu_u^2_4}%
\end{align}

Since $r<d$ this implies that $1+\frac{1}{d}-\frac{1}{r}>1+\frac{2}{d}%
-\frac{3}{r}$ and by interpolation we obtain that%
\begin{equation}
\left\Vert \varrho\right\Vert _{L^{\frac{1}{1+\frac{1}{d}-\frac{1}{r}}}}%
\leq\left\Vert \varrho\right\Vert _{L^{1}}^{1-q\left(  r,d\right)  }\left\Vert
\varrho\right\Vert _{L^{\frac{1}{1+\frac{2}{d}-\frac{3}{r}}}}^{q\left(
r,d\right)  },\label{estimate_interpolare_lebesque_urat}%
\end{equation}
with%
\[
q\left(  r,d\right)  =\frac{\frac{1}{r}-\frac{1}{d}}{\frac{3}{r}-\frac{2}{d}}.
\]
We obtain that%
\begin{align}
&  \Vert\varrho u\otimes u\Vert_{L^{\frac{r}{r-1}}}\Vert\nabla\Psi\Vert
_{L^{r}}+\varepsilon\left\Vert \nabla u\nabla\varrho\right\Vert _{L^{\frac
{1}{1+\frac{1}{d}-\frac{1}{r}}}}\left\Vert \Psi\right\Vert _{L^{\frac{dr}%
{d-r}}}\nonumber\\
&  \leq C\left(  \Omega,\left\Vert f,g\right\Vert _{L^{\infty}}\right)
(\left\Vert \varrho\right\Vert _{\frac{1}{1+\frac{2}{d}-\frac{3}{r}}}%
^{\frac{2}{r-1}q\left(  r,d\right)  +1}+1)\left\Vert \varrho\right\Vert
_{L^{\frac{r\gamma}{r-1}}}^{\frac{\gamma}{r-1}}.\label{rhu_u^2_44}%
\end{align}
Next, we observe that for all $\frac{3d}{d+2}\leq r\leq d$%
\[
\gamma>\frac{d\left(  r-1\right)  }{dr+2r-3d}\geq\frac{2}{r-1}q\left(
r,d\right)  +1.
\]
Moreover, using that
\begin{align*}
\frac{1}{1+\frac{2}{d}-\frac{3}{r}}  <\frac{r\gamma}{r-1} &\Leftrightarrow
1+\frac{2}{d}-\frac{3}{r}>\frac{r-1}{r\gamma},\\
& \Leftrightarrow1+\frac{2}{d}-\frac{3}{r}>\frac{r-1}{r\gamma},\\
&  \Leftrightarrow\frac{\left(  d+2\right)  r-3d}{d}>\frac{r-1}{\gamma},\\
&  \Leftrightarrow r>\frac{3d}{d+2}\text{ and }\gamma>\frac{d(r-1)}%
{\left(  d+2\right)  r-3d}%
\end{align*}
and combining this with estimates $\left(  \text{\ref{rhu_u^2_44}}\right)  $
and $\left(  \text{\ref{estimate_interpolare_lebesque_urat}}\right)  $ we
obtain the existence of some $\alpha_{1}\in\left(  0,1\right)  $ such that%
\begin{equation}
\Vert\varrho u\otimes u\Vert_{L^{\frac{r}{r-1}}}\Vert\nabla\Psi\Vert_{L^{r}%
}+\varepsilon\left\Vert \nabla u\nabla\varrho\right\Vert _{L^{\frac{1}%
{1+\frac{1}{d}-\frac{1}{r}}}}\left\Vert \Psi\right\Vert _{L^{\frac{dr}{d-r}}%
}\leq C\left(  \Omega,\left\Vert f,g\right\Vert _{L^{\infty}}\right)
(\left\Vert \varrho\right\Vert _{L^{\frac{r\gamma}{r-1}}}^{\frac{r\gamma}%
{r-1}-\alpha_{1}}+1).\label{rho_u^2_5}%
\end{equation}

We observe that%
\begin{equation}
\Vert\mathbb{S}\Vert_{L^{\frac{r}{r-1}}}\Vert\nabla\Psi\Vert_{L^{r}}\leq
C\left(  \Omega\right)  \left\Vert \nabla u\right\Vert _{L^{r}}^{r-1}%
\left\Vert \varrho\right\Vert _{L^{\frac{r\gamma}{r-1}}}^{\frac{\gamma}{r-1}%
}\leq C\left(  \Omega,\left\Vert f,g\right\Vert _{L^{\infty}}\right)
(\left\Vert \varrho\right\Vert _{L^{\frac{1}{1+\frac{1}{d}-\frac{1}{r}}}%
}+1)\left\Vert \varrho\right\Vert _{L^{\frac{r\gamma}{r-1}}}^{\frac{\gamma
}{r-1}}. \label{S_nabla_psi}%
\end{equation}
Using again $\left(  \text{\ref{estimate_interpolare_lebesque_urat}}\right)  $
we obtain the existence of some $\alpha_{2}\in\left(  0,1\right)  $ such that%
\begin{equation}
\Vert\mathbb{S}\Vert_{L^{\frac{r}{r-1}}}\Vert\nabla\Psi\Vert_{L^{r}}\leq
C\left(  \Omega,\left\Vert f,g\right\Vert _{L^{\infty}}\right)  (\left\Vert
\varrho\right\Vert _{L^{\frac{r\gamma}{r-1}}}^{\frac{r\gamma}{r-1}-\alpha_{2}%
}+1). \label{S_nabla_psi_2}%
\end{equation}

All the remaining terms from $\left(  \text{\ref{estimare_rho_cu_Bogo}%
}\right)  $ verify estimates of the type $\left(  \text{\ref{rho_u^2_5}%
}\right)  $-$\left(  \text{\ref{S_nabla_psi_2}}\right)  $ thus using Young's
inequality we conclude that
\[
\Vert\varrho\Vert_{L^{\frac{r\gamma}{r-1}}}^{\frac{r\gamma}{r-1}}\leq C\left(
M,\Omega,\left\Vert f,g\right\Vert _{L^{\infty}}\right)  .
\]
Thus, the last estimate of $\left(  \text{\ref{estimations_niveau_eps}%
}\right)  $ is proved. Returning back to $\left(
\text{\ref{estimation_energie_eps}}\right)  $ and using that $\frac{1}%
{1+\frac{1}{d}-\frac{1}{r}}<\frac{r\gamma}{r-1}$ we conclude, using
interpolation, the validity of the first estimate of $\left(
\text{\ref{estimations_niveau_eps}}\right)  $.

\underline{The case $r=d.$} From $\left(  \text{\ref{estimation_energie_eps}%
}\right)  $ it follows that for any $b>1$ that%
\begin{equation}
\left\Vert \nabla u\right\Vert _{L^{d}}\leq C\left(  \Omega,\left\Vert
f,g\right\Vert _{L^{\infty}},b\right)  (\left\Vert \varrho\right\Vert _{L^{b}%
}^{\frac{1}{d-1}}+1)\leq C\left(  \Omega,\left\Vert f,g\right\Vert
_{L^{\infty}},b\right)  (\left\Vert \varrho\right\Vert _{L^{b}}^{\frac{1}%
{d-1}}+1).\label{estimate_r=d}%
\end{equation}
Let us remark that using a density argument, we can test the momentum equation
with test functions belonging to $W_{0}^{1,d}\left(  \Omega\right)  $. Using
the Sobolev embedding and $\left(  \text{\ref{estimate_r=d}}\right)  $ we have
that for any $c>\frac{d}{d-1}.$%
\begin{equation}
\left\Vert \varrho u\otimes u\right\Vert _{L^{\frac{d}{d-1}}}\leq C\left(
\Omega,\left\Vert f,g\right\Vert _{L^{\infty}},b\right)  \left\Vert
\varrho\right\Vert _{L^{c}}(\left\Vert \varrho\right\Vert _{L^{b}}^{\frac
{2}{d-1}}+1).\label{estimate_r=d_1}%
\end{equation}
We consider arbitrary $a>1$ such that
\begin{equation}
\frac{1}{d}+\frac{1}{c}+\frac{1}{a}<1\label{restriction}%
\end{equation}
Again, using Proposition \ref{caldura} we have that
\begin{align}
\varepsilon\left\Vert \nabla u\nabla\varrho\right\Vert _{L^{\frac{1}{\frac
{1}{d}+\frac{1}{c}+\frac{1}{a}}}} &  \leq\varepsilon\left\Vert \nabla
\varrho\right\Vert _{L^{\frac{1}{\frac{1}{c}+\frac{1}{a}}}}\left\Vert \nabla
u\right\Vert _{L^{d}}\leq C\left(  \Omega\right)  \left\Vert \rho u\right\Vert
_{L^{\frac{1}{\frac{1}{c}+\frac{1}{a}}}}\left\Vert \nabla u\right\Vert
_{L^{d}}\nonumber\\
&  \leq C\left(  \Omega,a\right)  \left\Vert \varrho\right\Vert _{L^{c}%
}\left\Vert u\right\Vert _{L^{a}}\left\Vert \nabla u\right\Vert _{L^{d}%
}\nonumber\\
&  \leq C\left(  \Omega,\left\Vert f,g\right\Vert _{L^{\infty}},b,a\right)
C\left(  \Omega\right)  \left\Vert \varrho\right\Vert _{L^{c}}(\left\Vert
\varrho\right\Vert _{L^{b}}^{\frac{2}{d-1}}+1).\label{estimate_r=d_2}%
\end{align}
As before, we test the momentum equation with $\Psi=\mathcal{B}\left(
\varrho^{\frac{\gamma}{d-1}}-\frac{1}{|\Omega|}\int_{\Omega}\varrho
^{\frac{\gamma}{d-1}}\mathrm{d}x\right)  $, where $\mathcal{B}$ is the
Bogovskii operator see Proposition \ref{Bogo} from Appendix B. We have that
for any $q\in\lbrack1,\infty)$ :%
\begin{equation}
\left\Vert \Psi\right\Vert _{L^{q}}\leq C\left(  \Omega,q\right)  \left\Vert
\nabla\Psi\right\Vert _{L^{d}}\text{ and }\left\Vert \nabla\Psi\right\Vert
_{L^{d}}\leq C\left(  M,\Omega\right)  \left\Vert \varrho\right\Vert
_{L^{\frac{d\gamma}{d-1}}}^{\frac{\gamma}{d-1}}.\label{Bogo_r=d}%
\end{equation}
We obtain that for arbitrarily $c>\frac{d}{d-1}>b>1$ and $a>1$ such that
$\left(  \text{\ref{restriction}}\right)  $ holds%
\begin{align}
&  \Vert\varrho u\otimes u\Vert_{L^{\frac{r}{r-1}}}\Vert\nabla\Psi\Vert
_{L^{r}}+\varepsilon\left\Vert \nabla u\nabla\varrho\right\Vert _{L^{\frac
{1}{1+\frac{1}{d}-\frac{1}{r}}}}\left\Vert \Psi\right\Vert _{L^{\frac{dr}%
{d-r}}}\nonumber\\
&  \leq C\left(  \Omega,\left\Vert f,g\right\Vert _{L^{\infty}},b,a\right)
C\left(  \Omega\right)  \left\Vert \varrho\right\Vert _{L^{c}}(\left\Vert
\varrho\right\Vert _{L^{b}}^{\frac{2}{d-1}}+1)\left\Vert \varrho\right\Vert
_{L^{\frac{d\gamma}{d-1}}}\nonumber\\
&  \leq C\left(  \Omega,\left\Vert f,g\right\Vert _{L^{\infty}},b,a\right)
C\left(  \Omega\right)  \left\Vert \varrho\right\Vert _{L^{c}}(\left\Vert
\varrho\right\Vert _{L^{c}}^{\frac{2}{d-1}\frac{1-\frac{1}{b}}{1-\frac{1}{c}}%
}+1)\left\Vert \varrho\right\Vert _{L^{\frac{d\gamma}{d-1}}}%
.\label{estimate_r=d_4}%
\end{align}
since
\[
1+\lim_{b\rightarrow1}\frac{2}{d-1}\frac{1-\frac{1}{b}}{1-\frac{1}{c}%
}=1\text{, }\gamma>1
\]
we infer that we may chose $a,b,c$ in such a way that the last term from
$\left(  \text{\ref{estimate_r=d_4}}\right)  $ would be bounded by a multiple
of $\left\Vert \varrho\right\Vert _{L^{\frac{d\gamma}{d-1}}}^{\frac{d\gamma
}{\gamma-1}-\alpha}+1$ for some $\alpha>0$. The rest of the proof proceeds as
in the previous case.

\underline{The case $r>d.$} This case is much more simpler since we can obtain
a bound for the energy by just using the embedding of $W_{0}^{1,r}\left(
\Omega\right)  $ into $L^{\infty}\left(  \Omega\right)  $ in $\left(
\text{\ref{estimation_energie_eps}}\right)  $. We have directly that%
\[
\int_{\Omega}\mathbb{S(}u):\nabla u\mathrm{d}x+\frac{4\varepsilon}{\gamma}%
{\displaystyle\int_{\Omega}}
\left\vert \nabla\varrho^{\frac{\gamma}{2}}\right\vert ^{2}\mathrm{d}x\leq
C\left(  M,\Omega,\left\Vert f,g\right\Vert _{L^{\infty}}\right)  .
\]
As before, we obtain the existence of some $\alpha_{1}\in\left(  0,1\right)  $
such that
\[
\left\Vert \varrho u\otimes u\right\Vert _{L^{\frac{r}{r-1}}}\leq C\left(
M,\Omega\right)  \left\Vert \varrho\right\Vert _{L^{\frac{r\gamma}{r-1}%
-\alpha_{1}}}\left\Vert \nabla u\right\Vert _{L^{r}}^{2}\leq C\left(
M,\Omega,\left\Vert f,g\right\Vert _{L^{\infty}}\right)  \left\Vert
\varrho\right\Vert _{L^{\frac{r\gamma}{r-1}-\alpha_{1}}}.
\]
Next, for some $\alpha_{2}\in\left(  0,1\right)  $ we have that Proposition
\ref{caldura}
\[
\varepsilon\left\Vert \nabla u\nabla\varrho\right\Vert _{L^{1}}\leq C\left(
\Omega\right)  \left\Vert \nabla u\right\Vert _{L^{r}}\left\Vert \varrho
u\right\Vert _{L^{\frac{r}{r-1}}}\leq C\left(  M,\Omega,\left\Vert
f,g\right\Vert _{L^{\infty}}\right)  \left\Vert \varrho\right\Vert
_{L^{\frac{r\gamma}{r-1}-\alpha_{2}}}.
\]
The rest of the proof proceed as in the case $r<d$.
\end{proof}

In order to prove $\left(  \text{\ref{estimation_passage_limite_eps}}\right)
$ we first use Proposition \ref{caldura} to obtain that
\begin{equation}
\varepsilon\left\Vert \nabla\varrho\right\Vert _{L^{q_{1}^{\ast}}}\leq
C\left(  M,\Omega,\left\Vert f,g\right\Vert _{L^{\infty}}\right)
,\label{eps_0}%
\end{equation}
see $\left(  \text{\ref{definitie_q1_ast}}\right)  $ for the definition of
$q_{1}^{\ast}.$ When $\gamma\geq2$ then testing the continuity equation with
$\varrho$ we obtain that
\begin{align}
\varepsilon\int_{\Omega}|\nabla\varrho|^{2}\mathrm{d}x &  =-\frac{1}{2}%
\int_{\Omega}\varrho^{2}\operatorname{div}u\mathrm{d}x\leq C\left(
M,\Omega\right)  \Vert\varrho\Vert_{L^{\frac{2r}{r-1}}}^{2}\Vert
\operatorname{div}u\Vert_{L^{r}}\label{something_eps_1}\\
&  \leq C\left(  M,\Omega\right)  \Vert\varrho\Vert_{L^{\frac{r\gamma}{r-1}}%
}^{2}\Vert\operatorname{div}u\Vert_{L^{r}}\leq C\left(  M,\Omega,\left\Vert
f,g\right\Vert _{L^{\infty}}\right)  .\nonumber
\end{align}
We deduce that
\[
\varepsilon\left\Vert \nabla\varrho\right\Vert _{L^{2}}\leq C\left(
M,\Omega,\left\Vert f,g\right\Vert _{L^{\infty}}\right)  \sqrt{\varepsilon}.
\]
When $\gamma<2$ we have that
\[
\nabla\varrho=\frac{2}{\gamma}\varrho^{1-\frac{\gamma}{2}}\nabla\varrho
^{\frac{\gamma}{2}}%
\]
such that we end up with%
\begin{equation}
\varepsilon\left\Vert \nabla\varrho\right\Vert _{L^{p^{\ast}}}\leq
\sqrt{\varepsilon}\frac{2}{\gamma}\sqrt{\varepsilon}\left\Vert \nabla
\varrho^{\frac{\gamma}{2}}\right\Vert _{L^{2}}\left\Vert \varrho\right\Vert
_{L^{\frac{r\gamma}{r-1}}}^{1-\frac{\gamma}{2}}\leq C\left(  M,\Omega
,\left\Vert f,g\right\Vert _{L^{\infty}}\right)  \varepsilon^{\frac{1}{2}%
}\label{something_eps_2}%
\end{equation}
where%
\[
\frac{1}{p^{\ast}}=\frac{1}{2}+\frac{r-1}{r\gamma}\left(  1-\frac{\gamma}%
{2}\right)  <1\Leftrightarrow1-\frac{\gamma}{2}<\frac{1}{2}\frac{r\gamma
}{r-1}\Leftrightarrow1<\frac{\gamma}{2}\left(  1+\frac{r}{r-1}\right)  ,
\]
and the last inequality is always true.

It is easy to check that always
\[
p^{\ast}<2<q^{\ast}%
\]
and thus, interpolating between, on the one hand $\left(  \text{\ref{eps_0}%
}\right)  $ and on the other hand between $\left(  \text{\ref{something_eps_1}%
}\right)  $ for $\gamma\geq2$ or $\left(  \text{\ref{something_eps_2}}\right)
$ for $\gamma\in(1,2)$ we obtain that for any $\eta_{1}>0$ there exists
$\alpha_{1}\left(  \eta_{1}\right)  \in\left(  0,\frac{1}{2}\right)  $ such
that we have%
\begin{equation}
\lim_{\varepsilon\rightarrow0}\varepsilon\left\Vert \nabla\varrho\right\Vert
_{L^{q^{\ast}-\eta}}\leq C\left(  M,\Omega,\left\Vert f,g\right\Vert
_{L^{\infty}}\right)  \varepsilon^{\alpha_{1}\left(  \eta_{1}\right)
}.\label{inegalitate}%
\end{equation}
The estimates for $\varepsilon\nabla\varrho\nabla u$ and $\varepsilon
\nabla\varrho\otimes u$ are obtained by using $\left(  \text{\ref{inegalitate}%
}\right)  $ along with H\"{o}lder's inequality and the Sobolev embedding. We
leave the details as an exercise for the reader. This concludes the proof of
Theorem \ref{Existence_1_level}.
\end{proof}

\subsection{Limit passage with $\varepsilon\rightarrow0$%
\label{Limit passage with epsilon}}

In this section we prove that weak solutions are obtained as limits of
solutions for the reduced regularized systems $\left(  \text{\ref{system_eps}%
}\right)  $. We formalize the result in the following:

\begin{lemma}
\label{Lemma_conv}Consider $\varepsilon\in\left(  0,1\right)  $,$d\in\left\{
2,3\right\}  $ and $r,\gamma>1$ verifying $\left(  \text{\ref{cond_1}}\right)
$ or $\left(  \text{\ref{cond_2}}\right)  $. Let $\mathbb{S}:\mathbb{R}_{\mathrm{sym}}
^{d\times d}\rightarrow\mathbb{R}_{\mathrm{sym}}^{d\times d}$ be a continuous function
verifying the conditions (\ref{growth})-(\ref{monotonicity}). Fix
$M>0$ and $f,g\in(L^{\infty}\left(  \Omega\right)  )^{d}$. For all
$\varepsilon>0$ consider $(\varrho_{\varepsilon},u_{\varepsilon}%
)_{\varepsilon>0}\in W^{2,r}\left(  \Omega\right)  \times(W_{0}^{1,r}\left(
\Omega\right)  )^{d}$ a weak solutions of $\left(  \text{\ref{system_eps}%
}\right)  $ verifying the estimates (\ref{est_eps})-(\ref{estimation_passage_limite_eps}). Then up to a subsequence
we have
\[
\left\{
\begin{aligned}
\varrho_{\varepsilon}\rightarrow\varrho \quad &\text{strongly in} \quad L^{q}(\Omega)\text{
for any }q<\frac{r\gamma}{r-1},\\[0.5em]
u_{\varepsilon}\rightarrow u \quad &\text{strongly in} \quad (L^{r_{\ast}}(\Omega))^{d},\\[0.5em]
\mathbb{D}u_{\varepsilon}\rightarrow\mathbb{D}u \quad &\text{strongly in} \quad
(L^{p}(\Omega))^{d\times d}\text{ for any }p<r.
\end{aligned}
\right.
\]
In particular, the limit $(\varrho,u)$ satisfies the weak formulation of
$\left(  \text{\ref{main}}\right)  $.
\end{lemma}
\begin{proof}
    \underline{First step.} From the estimates $\left(
\text{\ref{estimations_niveau_eps}}\right)  $ we obtain the existence
of
\[ \begin{aligned}
    (\varrho, u) &\in L^{\frac{r\gamma}{r-1}(\Omega)}\times (W_0^{1,r}(\Omega))^d, \\
    \overline{\varrho^\gamma} &\in L^{\frac{r}{r-1}}(\Omega), \\
    \overline{\mS} &\in (L^{\frac{r}{r-1}}(\Omega))^{d\times d}
\end{aligned}\]
and positive measures%
\[
\overline{\mathbb{S}:\mathbb{D}u},\Theta\in\mathcal{M}%
(\Omega),
\]
such that up to a subsequence (denoted again by $\varepsilon$), we have that
\[ \left\{
\begin{aligned}
\varrho_{\varepsilon}\rightharpoonup\varrho \quad &\text{weakly in} \quad L^{\frac{r\gamma
}{r-1}}(\Omega),\\[0.5em]
\varrho_{\varepsilon}^{\gamma}\rightharpoonup\overline{\varrho^{\gamma}} \quad &\text{weakly in} \quad L^{\frac{r}{r-1}}(\Omega),\\[0.5em]
u_{\varepsilon}\rightarrow u \quad &\text{strongly in} \quad (L^{r}(\Omega))^{d},\\[0.5em]
\nabla u_{\varepsilon}\rightharpoonup\nabla u \quad &\text{weakly in} \quad (L^{r}%
(\Omega))^{d\times d},\\[0.5em]
\mathbb{S}_{\varepsilon}\rightharpoonup\overline{\mathbb{S}} \quad &\text{weakly in} \quad (L^{\frac{r}{r-1}}(\Omega))^{d\times d},\\[0.5em]
\mathbb{S}_{\varepsilon}:\mathbb{D}u_{\varepsilon}\rightharpoonup
\overline{\mathbb{S}:\mathbb{D}u} \quad &\text{weakly-$\ast$ in} \quad
\mathcal{M}(\Omega),\\[0.5em]
\varepsilon\left\vert \nabla\varrho^{\frac{\gamma}{2}}\right\vert
^{2}\rightharpoonup\Theta \quad &\text{weakly-$\ast$ in} \quad \mathcal{M}(\Omega),
\end{aligned}
\right.
\]
where $\mathcal{M}(\Omega)=\left(  \mathcal{C}\left(  \overline{\Omega
}\right)  \right)  ^{\prime}$ is the space, of Radon measures on $\Omega$ while
$\mathcal{C}\left(  \overline{\Omega}\right)  $ is the set of continuous
functions on $\overline{\Omega}$. From the monotonicity of the stress tensor
we also have the inequality
\[
\overline{\mathbb{S}:\mathbb{D}u}\geq\overline{\mathbb{S}}:\mathbb{D}u.
\]
Moreover, using $\left(  \text{\ref{estimation_passage_limite_eps}}\right)  $
we have that
\[
\lim_{\varepsilon\rightarrow0}\varepsilon\int_{\Omega}|\nabla\varrho
_{\varepsilon}\cdot\nabla u_{\varepsilon}|\;\mathrm{d}x=0.
\]
Passing to the limit in the equations verified by $\left(  \varrho
_{\varepsilon},u_{\varepsilon}\right)  $ we obtain that
\begin{equation}
\left\{
\begin{array}
[c]{r}%
\mathrm{div\,}(\varrho u)=0,\\
\mathrm{div\,}(\varrho u\otimes u)-\mathrm{div\,}\overline{\mathbb{S}}%
+\nabla\overline{\varrho^{\gamma}}=\varrho f+g.
\end{array}
\right.  \label{limit}%
\end{equation}

\underline{Second step.} Having (\ref{limit}), it is now enough to show that in fact $\overline{\mathbb{S}}=\mathbb{S}$ and
$\overline{\varrho^{\gamma}}=\varrho^{\gamma}$. It is also equivalent
to the strong convergence of the respective sequences. To do that, we first show the following:

\begin{proposition}
\label{Prop_identity} The identity
\begin{equation}
\frac{-1}{\gamma-1}\int_{\Omega}\overline{\varrho^{\gamma}}u\cdot\nabla
\varphi \;\mathrm{d}x+\int_{\Omega}\overline{\varrho^{\gamma}}\mathrm{div\,}%
u\,\varphi \;\mathrm{d}x+\int_{\Omega}\varphi\;\mathrm{d}\left[  \overline
{\mathbb{S}:\mathbb{D}u}-\overline{\mathbb{S}}:\mathbb{D}u\right]  +\frac
{4}{\gamma}\int_{\Omega}\varphi\;\mathrm{d}\Theta=0\text{,}\label{identity_rho}%
\end{equation}
holds for all $\varphi\in C^{1}\left(  \overline{\Omega}\right)  $.
\end{proposition}
\begin{proof}
In view of Proposition \ref{Prop_ren2} from the Appendix B., the proof of
which is a minor modification of Lemma $3.3.$ page $160$ of
\cite{NovotnyStraskraba2004}, $\left(  \varrho,u\right)  $ constructed above
verify the renormalized continuity equation:%
\begin{equation}
-
{\displaystyle\int_{\Omega}}
b(\varrho)u\cdot\nabla\varphi\;\mathrm{d}x+
{\displaystyle\int_{\Omega}}
(\left(  b^{\prime}(\varrho)-\varrho b\left(  \varrho\right)  \right)
\mathrm{div\,}u\cdot\varphi\;\mathrm{d}x=0,\label{renorm_up_to_boundary1}%
\end{equation}
for all $\varphi\in\mathcal{C}^{1}\left(  \overline{\Omega}\right)  $. Given
the fact that $\left(  \varrho,u\right)  \in L^{\frac{r\gamma}{r-1}}%
(\Omega)\times(W_{0}^{1,r}(\Omega))^{d}$ we can write $\left(
\text{\ref{renorm_up_to_boundary1}}\right)  $ also for $b\left(
\varrho\right)  =\varrho^{\gamma}$ which gives:
\begin{equation}
-
{\displaystyle\int_{\Omega}}
\varrho^{\gamma}u\cdot\nabla\varphi\;\mathrm{d}x+\left(  \gamma-1\right)
{\displaystyle\int_{\Omega}}
\varrho^{\gamma}\mathrm{div\,}u\cdot\varphi\;\mathrm{d}%
x=0,\label{renorm_pressure}%
\end{equation}
for all $\varphi\in\mathcal{C}^{1}\left(  \overline{\Omega}\right)  $. Observe
that the $W^{2,r}$-regularity of $\varrho_{\varepsilon}$ along with a density argument enables us to use $\gamma\varrho_{\varepsilon
}^{\gamma-1}\varphi$ as a test function in the first equation of $\left(
\text{\ref{system_eps}}\right)  $ for all $\varphi\in\mathcal{C}^{1}\left(
\overline{\Omega}\right)  $ and as such to obtain that
\begin{equation}
-
{\displaystyle\int_{\Omega}}
\varrho_{\varepsilon}^{\gamma}u_{\varepsilon}\cdot\nabla\varphi\;\mathrm{d}x+\left(
\gamma-1\right)
{\displaystyle\int_{\Omega}}
\varrho_{\varepsilon}^{\gamma}\mathrm{div\,}u_{\varepsilon}\cdot
\varphi\;\mathrm{d}x=-\varepsilon\gamma
{\displaystyle\int_{\Omega}}
\varrho_{\varepsilon}^{\gamma-1}\nabla\varrho_{\varepsilon}\nabla
\varphi\;\mathrm{d}x-4\varepsilon\frac{\gamma-1}{\gamma}
{\displaystyle\int_{\Omega}}
|\nabla\varrho_{\varepsilon}^{\gamma/2}|^{2}\varphi\;\mathrm{d}%
x\label{renorm_epsilon}%
\end{equation}
for all $\varphi\in\mathcal{C}^{1}\left(  \overline{\Omega}\right)  $. The
fact that $u_{\varepsilon}\in(W_{0}^{1,r}\left(  \Omega\right)  )^{d}$ along
with the restrictions $\left(  \text{\ref{restriction_r_gamma_1}}\right)
$-$\left(  \text{\ref{restriction_r_gamma_2}}\right)  $ ensure that we can
test the momentum equation by $\varphi u_{\varepsilon}$ and, using a density
argument, the first equation of $\left(  \text{\ref{system_eps}}\right)  $
with $\frac{\left\vert u^{\varepsilon}\right\vert ^{2}}{2}\varphi$ for any
$\varphi\in\mathcal{C}^{1}\left(  \overline{\Omega}\right)  $. Taking the sum
and using $\left(  \text{\ref{renorm_epsilon}}\right)  $, we derive the local
energy equality
\begin{align}\label{energy_eq_eps}
\int_{\Omega}
\left(  -\frac{1}{2}\varrho_{\varepsilon}|u_{\varepsilon}|^{2}u_{\varepsilon} -\frac{\varepsilon}{2}|u_\varepsilon|^2\nabla\varrho_\varepsilon
+\mathbb{S}_{\varepsilon}u_{\varepsilon}-\frac{\gamma}{\gamma-1}%
\varrho_{\varepsilon}^{\gamma}u_{\varepsilon}+\frac{\varepsilon\gamma}%
{\gamma-1}\varrho_{\varepsilon}^{\gamma/2}\nabla\varrho_{\varepsilon}%
^{\gamma/2}\right)  \nabla\varphi\text{ }\mathrm{d}x \\
+\int_{\Omega}%
\mathbb{S}_{\varepsilon}  :\mathbb{D}u_{\varepsilon}\cdot\varphi\text{
}\mathrm{d}x\nonumber
 =-\frac{4\varepsilon}{\gamma}%
{\displaystyle\int_{\Omega}}
|\nabla\varrho_{\varepsilon}^{\gamma/2}|^{2}\mathrm{d}x %
\end{align}
for all $\varphi\in\mathcal{C}^{1}\left(  \overline{\Omega}\right)  $. By virtue of strong convergence of
$u_{\varepsilon}$, estimates (\ref{estimation_passage_limite_eps}) and
\[
\varepsilon\int_{\Omega}\varrho_{\varepsilon}^{\gamma/2}|\nabla\varrho
_{\varepsilon}^{\gamma/2}|\;\mathrm{d}x\leq\varepsilon\Vert\varrho
_{\varepsilon}\Vert_{L^{\gamma}}^{\gamma/2}\Vert\nabla\varrho_{\varepsilon
}^{\gamma/2}\Vert_{L^{2}}\rightarrow0,
\]
after passing to the limit in $\left(  \text{\ref{energy_eq_eps}}\right)  $ we
obtain
\begin{equation}%
{\displaystyle\int_{\Omega}}
\left(  -\frac{1}{2}\varrho|u|^{2}u+\mathbb{\bar{S}}u-\frac{\gamma}{\gamma-1}%
\overline{\varrho^{\gamma}}u\right)  \nabla\varphi\text{ }\mathrm{d}%
x+\int_{\Omega}\varphi\text{ }\mathrm{d}\left[  \overline{\mathbb{S}%
:\mathbb{D}u}\right]  =-\frac{4}{\gamma}\int_{\Omega}\varphi\text{ }%
\mathrm{d}\Theta,\label{st_energy1}%
\end{equation}
for all $\varphi\in\mathcal{C}^{1}\left(  \overline{\Omega}\right)  $. Now
analogically testing the momentum equation in (\ref{limit}) by $\varphi u$,
$\varphi\in\mathcal{C}^{1}\left(  \overline{\Omega}\right)  $ we get
\begin{equation}%
{\displaystyle\int_{\Omega}}
\left(  -\varrho|u|^{2}u+\mathbb{\bar{S}}u-\overline{\varrho^{\gamma}%
}u\right)  \nabla\varphi\text{ }\mathrm{d}x-\int_{\Omega}\overline
{\varrho^{\gamma}}\mathrm{div\,}u\cdot\varphi\text{ }\mathrm{d}x+\int_{\Omega
}\mathbb{\bar{S}}:\mathbb{D}u\varphi\text{ }\mathrm{d}x=0.\label{st_energy2}%
\end{equation}
for all $\varphi\in\mathcal{C}^{1}\left(  \overline{\Omega}\right)  $.
Subtracting $\left(  \text{\ref{st_energy2}}\right)  $ from $\left(
\text{\ref{st_energy1}}\right)  $, we get to the desired conclusion.
\end{proof}
\underline{Third step.} In order to obtain strong convergence for the $\nabla
u_{\varepsilon}$, we prove the following

\begin{proposition}
\label{propo_limit}
Consider $\gamma>1,$ $u\in(W_{0}^{1,r}\left(
\Omega\right)  )^{d}$, $b\in L^{\frac{r}{r-1}}\left(  \Omega\right)$ nonnegative and a positive measure $\mathcal{R}\in\mathcal{M}\left(  \Omega\right)  $ such that the identity
\begin{equation}
\frac{-1}{\gamma-1}\int_{\Omega}b u\nabla\varphi\text{ }\mathrm{d}%
x+\int_{\Omega}b\mathrm{div\,}u\cdot\varphi\text{ }\mathrm{d}%
x+\int_{\Omega}\varphi\text{ }\mathrm{d}\mathcal{R}=0\text{,}\label{identity}%
\end{equation}
holds for all $\varphi\in\mathcal{C}^{1}(\overline{\Omega})$. Then, for all
$\eta>0$ there exists a compact set $\Omega_{\eta}\subset\Omega$ such that
$\left\vert \Omega\backslash\Omega_{\eta}\right\vert \leq\eta$ and $\mathcal{R}\left(
\Omega_{\eta}\right)  =0$.
\end{proposition}
\begin{proof}
Since $\left(  \text{\ref{identity}}\right)  $ holds with $\varphi
\in\mathcal{C}^{1}(\overline{\Omega})$, by denoting by $\tilde{b},\tilde
{u},\tilde{\mathcal{R}}$ the extensions by $0$ outside $\Omega$, i.e.%
\[
\forall\varphi\in\mathcal{C}_{0}(\mathbb{R}^{d}):\left\langle \tilde{\mathcal{R}
},\varphi\right\rangle =\left\langle \mathcal{R},\varphi_{|\overline{\Omega}%
}\right\rangle,
\]
we have that%
\begin{equation}
\operatorname{div}(\tilde{b}\tilde{u})+\left(  \gamma-1\right)
\tilde{b}\operatorname{div}\tilde{u}+\left(  \gamma-1\right)\tilde{\mathcal{R}}=0\text{ in }%
\mathcal{E}^{\prime}(\mathbb{R}^{d}).\label{identity_delta}%
\end{equation}
In the following, in order to ease the reading we will drop the tilde overscript.

For a standard mollifier $\omega_{\alpha}=\frac{1}{\alpha^{d}}\omega\left(
\frac{\cdot}{\alpha}\right)  $, $\alpha>0,$ with $\omega\in\mathcal{C}%
^{\infty}\left(  \mathbb{R}^{d}\right)  $ radial function with support in the
unit ball, we take
\[
b_{\alpha}=\omega_{\alpha}\ast b
\]
and
\[
r_{\alpha}=\mathrm{div\,}[b_{\alpha}u-\omega_{\alpha}\ast(b
u)]+(\gamma-1)[b_{\alpha}\mathrm{div\,}u-\omega_{\alpha}\ast
(b\mathrm{div\,}u)].
\]
Then from $\left(  \text{\ref{identity_delta}}\right)  $ we get
\begin{equation}
\mathrm{div\,}(b_{\alpha}u)+(\gamma-1)b_{\alpha}\mathrm{div\,}%
u+(\gamma-1)\omega_{\alpha}\ast \mathcal{R}=r_{\alpha}.\label{avant_multiplicare}%
\end{equation}
Since $\mathcal{R}\in\mathcal{E}^{\prime}(\mathbb{R}^{d})$ is a compactly supported
distribution on $\mathbb{R}^{n}$ and $\omega_{\alpha}\in\mathcal{C}%
_{c}^{\infty}$ we have that $\omega_{\alpha}\ast \mathcal{R}\in\mathcal{C}_{c}^{\infty
}\left(  \mathbb{R}^{d}\right)  $ and the above relation holds almost everywhere on
$\Omega$. Let $h\in\left(  0,1\right)  $ and multiply $\left(
\text{\ref{avant_multiplicare}}\right)  $ by $\frac{1}{\gamma}(b_{\alpha
}+h)^{1/\gamma-1}$. We obtain that
\[
\mathrm{div\,}\left(  (b_{\alpha}+h)^{1/\gamma}u\right)  +\frac{\gamma
-1}{\gamma}(b_{\alpha}+h)^{1/\gamma-1}\omega_{\alpha}\ast \mathcal{R}
=(b_{\alpha}+h)^{1/\gamma-1}h\mathrm{div\,}u+\frac{1}{\gamma}%
(b_{\alpha}+h)^{1/\gamma-1}r_{\alpha}%
\]
and after integration over $\mathbb{R}^d$ we get
\begin{align*}
\frac{\gamma-1}{\gamma}\int_{\Omega}\omega_{\alpha}\ast(b_{\alpha
}+h)^{1/\gamma-1}\mathrm{d} \mathcal{R} &  =\int_{\Omega}(b_{\alpha}%
+h)^{1/\gamma-1}h\mathrm{div\,}u\text{ }\mathrm{d}x+\frac{1}{\gamma}%
\int_{\mathbb{R}^d}(b_{\alpha}+h)^{1/\gamma-1}r_{\alpha}\;\mathrm{d}x.\\
&  \leq h^{\frac{1}{\gamma}}\int_{\Omega}\left\vert \mathrm{div\,}u\right\vert
\mathrm{d}x+\frac{1}{\gamma}\int_{\mathbb{R}^d}(b_{\alpha}+h)^{1/\gamma
-1}r_{\alpha}\;\mathrm{d}x
\end{align*}
Now we want to pass to the limit with $\alpha\rightarrow0$. By Friedrich's
lemma \ref{Prop_ren1} we know that $r_{\alpha}\rightarrow0$ in a suitable
$L^{p}$ space, therefore the last term converges to $0$. Thus, we obtain that for all $h>0$
\begin{equation}
\limsup\limits_{\alpha\rightarrow0}\int_{\Omega}\omega_{\alpha}\ast
(b_{\alpha}+h)^{1/\gamma-1}\mathrm{d}\mathcal{R}\leq h^{\frac{1}{\gamma}}%
\int_{\Omega}\left\vert \mathrm{div\,}u\right\vert \mathrm{d}%
x.\label{identity_avant passer limite}%
\end{equation}

Next, since up to a subsequence
\[
\omega_{\alpha}\ast(b_{\alpha}+h)^{1/\gamma-1}\rightarrow(b
+h)^{1/\gamma-1}\quad\text{a.e. in }\Omega,
\]
by Egorov's theorem and inner regularity of the Lebesgue measure there exists
a compact subset $\Omega_{\eta,h}\subset\Omega$ such that $|\Omega
\setminus\Omega_{\eta,h}|<\eta2^{-\frac{1}{h}}$ and
\[
\omega_{\alpha}\ast(b_{\alpha}+h)^{1/\gamma-1}\rightarrow(b+h)^{1/\gamma-1}%
\quad\text{uniformly in }\Omega_{\eta,h}.
\]
In particular, for all $\eta_{1}>0$, there exists
$\alpha\left(  \eta_{1}\right)  >0$ such that for $\alpha<\alpha(\eta_1)$
\[
(b+h)^{1/\gamma
-1}\leq\omega_{\alpha}\ast(b_{\alpha}+h)^{1/\gamma-1}+\eta_{1} \quad \text{on} \quad \Omega_{\eta,h}.
\]

Fix $n\in\mathbb{N}$, $\eta_{1}>0$ and consider $\alpha\left(  \eta_{1}\right)  $ as above. On the set $\left\{  x\in\Omega:b<n\right\}  \cap\Omega_{\eta,h}$ we
have that%
\[
(n+1)^{1/\gamma-1}\leq(b+h)^{1/\gamma-1}\leq\omega_{\alpha}\ast
(b_{\alpha}+h)^{1/\gamma-1}+\eta_{1}%
\]
thus, we have that for all $x\in\Omega$%
\[
(n+1)^{1/\gamma-1}\mathbbm{1}_{\left\{  b<n\right\}  \cap\Omega_{\eta,h
}}\leq\omega_{\alpha}\ast(b_{\alpha}+h)^{1/\gamma-1}+\eta_{1}.
\]
Since $(b+h)^{1/\gamma-1}$ is obtained by uniform convergence of continuous functions on $\Omega_{\eta,h}$, we conclude that the set $\left\{  x\in\Omega:b<n\right\} \cap\Omega_{\eta,h}$ is open in $\Omega_{\eta,h}$ with the relative topology and thus it is the intersection of an open set of $\mathbb{R}^{d}$ with the compact set $\Omega_{\eta,h}$ and therefore it is $\mathcal{R}$-measurable. By integration we have that%
\[
(n+1)^{1/\gamma-1}\mathcal{R}\left(  \left\{  \delta<n\right\}  \cap\Omega_{\eta,h
}\right)  \leq\int_{\Omega}\omega_{\alpha}\ast(b_{\alpha}+h)^{1/\gamma
-1}\mathrm{d}\mathcal{R}+\eta_{1}\mathcal{R}\left(  \Omega\right)  .
\]
Taking $\liminf_{\alpha\to 0}$, since $\eta_{1}>0$ is
arbitrary, we obtain that%
\begin{equation}
(n+1)^{1/\gamma-1}\mathcal{R}\left(  \left\{  b<n\right\}  \cap\Omega_{\eta,h
}\right)  \leq\liminf\limits_{\alpha\rightarrow0}\int_{\Omega}\omega_{\alpha
}\ast(b_{\alpha}+h)^{1/\gamma-1}d\mathcal{R}.\label{inegalitate_masuri}%
\end{equation}
Combining $\left(  \text{\ref{inegalitate_masuri}}\right)  $ with $\left(
\text{\ref{identity_avant passer limite}}\right)  $, we obtain%
\begin{equation}
(n+1)^{1/\gamma-1}\mathcal{R}\left(  \left\{  b<n\right\}  \cap\Omega_{\eta,h
}\right)  \leq \frac{\gamma}{\gamma-1}h^{\frac{1}{\gamma}}\int_{\Omega}\left\vert \mathrm{div\,}%
u\right\vert \mathrm{d}x.
\label{avant_passer}
\end{equation}
Now, for each $j\in\mathbb{N}^*$ consider the above construction with $h=2^{-j}$, that is let $\Omega_{\eta,2^{-j}}$ be compact such that $|\Omega \setminus\Omega_{\eta,2^{-j}} |\leq\eta2^{-j}$, verifying \eqref{avant_passer} with $h=2^{-j}$, and consider $\Omega_\eta:=\bigcap_{j\in\mathbb{N}^*}\Omega_{\eta,2^{-j}}$. Then, $\Omega_\eta$ is compact and $|\Omega \setminus\Omega_\eta |\leq\eta$. 
From \eqref{avant_passer} it follows that
\[ (n+1)^{1/\gamma-1}\mathcal{R}\left(  \left\{  b<n\right\}  \cap\Omega_{\eta
}\right)  \leq \frac{\gamma}{\gamma-1}h^{\frac{1}{\gamma}}\int_{\Omega}\left\vert \mathrm{div\,}%
u\right\vert \mathrm{d}x \]
and therefore, taking $h\to\infty$ and owing to the positivity of $\mathcal{R}$, we obtain
\[
\mathcal{R}\left(  \left\{  b<n\right\}  \cap\Omega_{\eta}\right)  =0.
\]
In consequence, since $n$ was arbitrary,
\[
\mathcal{R}\left(  \Omega_{\eta}\right)  =0.
\]
This concludes the proof of Proposition \ref{propo_limit}.
\end{proof}

\underline{Fourth step : strong convergence of $\nabla u_{\varepsilon}$ and
$\varrho_{\varepsilon}$.} Owing to Propositions \ref{Prop_identity} and
\ref{propo_limit}, for all $\eta>0$ we obtain the existence of of a compact set
such that $\left\vert \Omega\backslash\Omega_{\eta}\right\vert \leq\eta$ and%
\[
\left[  \overline{\mathbb{S}:\mathbb{D}u}-\overline{\mathbb{S}}:\mathbb{D}%
u\right]  \left(  \Omega_{\eta}\right)  =0.
\]
Owing to monotonicity condition verified by $\mathbb{S}$ we have that%
\[
\limsup_{\varepsilon\rightarrow0}\left\Vert \mathbb{D}u_{\varepsilon
}-\mathbb{D}u\right\Vert _{L^{\ell}\left(  \Omega_{\eta}\right)  }^{\ell}%
\leq\left[  \overline{\mathbb{S}:\mathbb{D}u}-\overline{\mathbb{S}}%
:\mathbb{D}u\right]  \left(  \Omega_{\eta}\right)  =0.
\]
Then, for all $p<r$
\begin{align*}
\Vert\mathbb{D}u_{\varepsilon}-\mathbb{D}u\Vert_{L^{p}(\Omega)} &  \leq
\Vert\mathbb{D}u_{\varepsilon}-\mathbb{D}u\Vert_{L^{p}(\Omega_{\eta})}%
+\Vert\mathbb{D}u_{\varepsilon}-\mathbb{D}u\Vert_{L^{p}(\Omega\setminus
\Omega_{\eta})}\\
&  \leq C(\Omega)\Vert\mathbb{D}u_{\varepsilon}-\mathbb{D}u\Vert_{L^{r}%
(\Omega_{\eta})}+(\Vert\mathbb{D}u_{\varepsilon}\Vert_{L^{r}(\Omega)}%
+\Vert\mathbb{D}u\Vert_{L^{r}(\Omega)})|\Omega\setminus\Omega_{\eta}%
|^{\frac{r-p}{pr}}%
\end{align*}
and taking $\limsup_{\varepsilon\rightarrow0}$ on both sides we obtain
\[
\limsup_{\varepsilon\rightarrow0}\Vert\mathbb{D}u_{\varepsilon}-\mathbb{D}%
u\Vert_{L^{p}(\Omega)}\leq C\eta^{\frac{r-p}{pr}}.
\]
Since $\eta$ is arbitrary, this finally gives us
\[
\mathbb{D}u_{\varepsilon}\rightarrow\mathbb{D}u\quad\text{in}\quad
L^{p}(\Omega).
\]
Using Korn's inequality, we obtain the convergence of the whole gradient. Since up to a subsequence $\nabla u$ converges almost everywhere with respect to each coordinate, owing to the continuity of $\mathbb{S}$ we have that $\mathbb{S}%
_{ij}\left(  \mathbb{D}u_{\varepsilon}\right)  \rightarrow\mathbb{S}%
_{ij}\left(  \mathbb{D}u\right)  $.\footnote{We denote $(\mS(\D u))_{ij}=\mS_{ij}(\D  u)$, $i,j=1,\dots,d$}

For the convergence of $\varrho_{\varepsilon}$, let us localize the momentum
equation and then take the divergence of both sides. For any $\varphi\in
C_{0}^{\infty}(\Omega)$, we obtain that
\[
\Delta(\varphi\varrho_{\varepsilon}^{\gamma}) = -\mathrm{div\,}(\varphi
\varrho_{\varepsilon}u_{\varepsilon}\nabla u_{\varepsilon}) + \mathrm{div\,}%
(\varphi\mathrm{div\,}\mathbb{S}_{\varepsilon}) -\varepsilon\ddiv(\varphi\nabla\varrho_\varepsilon\nabla u_\varepsilon) + \ddiv(\varphi(\varrho_\varepsilon f_\varepsilon+g_\varepsilon)) + \varrho_{\varepsilon
}^{\gamma}\Delta\varphi+ \nabla\varphi\cdot\nabla\varrho_{\varepsilon}%
^{\gamma}.
\]

Therefore
\[\begin{aligned}
\varphi\varrho_{\varepsilon}^{\gamma}= & -\Delta^{-1}\mathrm{div\,}%
(\varphi\varrho_{\varepsilon}u_{\varepsilon}\nabla u_{\varepsilon}%
)+\Delta^{-1}\mathrm{div\,}(\varphi\mathrm{div\,}\mathbb{S}_{\varepsilon
}) - \varepsilon\Delta^{-1}\ddiv(\varphi\nabla\varrho\nabla u)+ \Delta^{-1}\ddiv(\varphi(\varrho f+g)) \\
&+\Delta^{-1}(\varrho_{\varepsilon}^{\gamma}\Delta\varphi)+\Delta^{-1}%
(\nabla\varphi\cdot\nabla\varrho_{\varepsilon}^{\gamma})
\end{aligned}\]
and integrating by parts the right hand side, we get
\begin{equation}\label{p}
\begin{aligned}
\varphi\varrho_{\varepsilon}^{\gamma}(t,x) =& -\int_{\Omega}\nabla K(x-y)\varphi
(y)\varrho_{\varepsilon}(y)u_{\varepsilon}(y)\nabla u_{\varepsilon
}(y)\;\mathrm{d}y+p.v.\int_{\Omega}\varphi(y)\mathbb{S}_{\varepsilon
}(y):\nabla^{2}K(x-y)\;\mathrm{d}y \\
&-\int_{\Omega}\mathbb{S}_{\varepsilon}(y):\nabla\varphi(y)\otimes\nabla
K(x-y)\;\mathrm{d}y
+\varepsilon\int_\Omega \nabla K(x-y)\varphi(y)\nabla\varrho_\varepsilon(y)\nabla u_\varepsilon(y)\;\dd y \\
&- \int_\Omega \nabla K(x-y)\varphi(y)(\varrho_\varepsilon(y)f_\varepsilon(y)+g_\varepsilon(y))\;\dd y +\int_{\Omega}K(x-y)\varrho_{\varepsilon}^{\gamma
}(y)\Delta\varphi(y)\;\mathrm{d}y \\
&+\int_{\Omega}\varrho_{\varepsilon}^{\gamma}(y)\nabla K(x-y)\cdot
\nabla\varphi(y)\;\mathrm{d}y-\int_{\Omega}\varrho_{\varepsilon}^{\gamma
}(y)K(x-y)\Delta\varphi(y)\;\mathrm{d}y,
\end{aligned}
\end{equation}
where $K(x)=\frac{c_{d}}{|x|^{d-2}}$ is the Newtonian potential with the usual
modification if $d=2$. From the elliptic estimates and the strong convergence
of $\mathbb{S}_{\varepsilon}$, it follows that all terms in the right hand
side of (\ref{p}) converge strongly in $L^{p}(\Omega)$ for some $p<r$. Indeed,
since $\nabla^{2}K$ is a Calder\'{o}n-Zygmund kernel, the convergence of the
terms containing the stress tensor follows from the strong convergence of
$\mathbb{S}_{\varepsilon}$. On the other hand, since $\varrho_{\varepsilon
}u_{\varepsilon}\nabla u_{\varepsilon}$ and $\varrho_{\varepsilon}^{\gamma}$
are uniformly bounded in $L^{q}$ for some $q>1$, the rest of the terms in
(\ref{p}) are bounded in $W^{1,q}$ and thus they converge strongly in
$L^{p}(\Omega)$ for some $p>1$. In conclusion,
\[
\varphi\varrho_{\varepsilon}^{\gamma}\rightarrow\varphi\overline
{\varrho^{\gamma}}\quad\text{in}\quad L^{p}(\Omega)
\]
as well, and in consequence $\varrho_{\varepsilon}^{\gamma}$ converges
strongly in $L^{1}(\Omega)$. Finally, using the fact that $\varrho
_{\varepsilon}\rightharpoonup\varrho$ and
\[
\varrho_{\varepsilon}=(\varrho_{\varepsilon}^{\gamma})^{1/\gamma}%
\rightarrow\overline{\varrho^{\gamma}}^{1/\gamma},
\]
we get that $\overline{\varrho^{\gamma}}=\varrho^{\gamma}$. This concludes the
proof of Lemma \ref{Lemma_conv}, and in consequence Theorem \ref{th_st}.
\end{proof}
The arguments and Proposition \ref{Prop_identity} remain valid,
\textit{mutatis mutandis}, in the case of Theorem \ref{Stability}. 

\section{Proof of Theorem \ref{Viscoplastic_is_fantastic}%
\label{Proof of Theorem Viscoplastic}}

In this section, we prove Theorem \ref{Viscoplastic_is_fantastic}. We will use
the same construction for the singular part of the viscous stress tensor as
the one found in \cite{malek2005herschel}. A weak solution will be obtained as
a limit of a subsequence of solutions for the following family of approximate
systems:%
\begin{equation}
\left\{
\begin{array}
[c]{ll}%
\alpha\varrho+\operatorname{div}\left(  \varrho u\right)  =\alpha
\check{\varrho} & \text{in }\Omega,\\
\beta\rho u+\operatorname{div}\left(  \varrho u\otimes u\right)
-\operatorname{div}(\left\vert \mathbb{D}u\right\vert ^{r-2}\mathbb{D}%
u)-\operatorname{div}\mathbb{P}_{\varepsilon}+\nabla\varrho^{\gamma}=\rho
f+g & \text{in }\Omega,
\end{array}
\right.  \label{approx_viscoplastic}%
\end{equation}
where%
\begin{equation}
\mathbb{P}_{\varepsilon}=\mathbb{P}_{\varepsilon}\left(  \mathbb{D}u\right)
=\tau^{\ast}g_{\varepsilon}\left(  \left\vert \mathbb{D}u\right\vert \right)
\mathbb{D}u\label{definition_P_eps}%
\end{equation}
with $g_{\varepsilon}\in C^{1}[0,+\infty)$ non increasing such that%
\begin{equation}
g_{\varepsilon}\left(  s\right)  =\left\{
\begin{array}
[c]{l}%
\frac{1}{\varepsilon}\text{ for }s\in\lbrack0,\frac{\varepsilon}{2}],\\
\frac{1}{s}\text{ for }\frac{3\varepsilon}{2}\leq s,
\end{array}
\right.  \label{definitie_g_eps}%
\end{equation}
along with%
\[
g_{\varepsilon}\left(  s\right)  \leq\min\left\{  \frac{1}{\varepsilon}%
,\frac{1}{s}\right\}  \text{ and }-\frac{4}{9\varepsilon^{2}}\leq
g_{\varepsilon}^{\prime}\left(  s\right)  \leq0\text{ }\forall s\in
\lbrack0,\infty).
\]
These properties ensure that $\mathbb{P}_{\varepsilon}:\mathbb{R}^{d\times
d}\rightarrow\mathbb{R}^{d\times d}$ is continuous for all $A\in
\mathbb{R}^{d\times d}$ and that
\begin{equation}
\left\{
\begin{array}
[c]{l}%
\mathbb{P}_{\varepsilon}\left(  A\right)  :A\geq0,\\
\left\vert \mathbb{P}_{\varepsilon}\left(  A\right)  \right\vert \leq
\tau^{\ast},\\
\left(  \mathbb{P}_{\varepsilon}\left(  A\right)  -\mathbb{P}_{\varepsilon
}\left(  B\right)  \right)  \left(  A-B\right)  \geq0,
\end{array}
\right.  \label{P_eps}%
\end{equation}
see \cite{malek2005herschel} page $1849$.

Consider $d\in\left\{  2,3\right\}  ,$ $r,\gamma>1$ such that one of $\left(
\text{\ref{restriction_r_gamma_1}}\right)  $ or $\left(
\text{\ref{restriction_r_gamma_2}}\right)  $ hold true. Fix $M>0$ and
$f,g\in(L^{\infty}\left(  \Omega\right)  )^{d}$. The construction of
weak-solutions for the system $\left(  \text{\ref{approx_viscoplastic}%
}\right)  $ is obtained with exactly the same arguments as above with a minor
modification which requires that $\gamma\in(1,2]$. Moreover, if we analyze the
proof of Theorem \ref{th_st}, see Section \ref{Uniform_Estimates} it
transpires that
\[
\left\Vert \varrho_{\varepsilon}\right\Vert _{L^{\frac{r\gamma}{r-1}}%
}+\left\Vert \nabla u_{\varepsilon}\right\Vert _{L^{r}}+\left\Vert
\mathbb{P}_{\varepsilon}\left(  \mathbb{D}u_{\varepsilon}\right)  \right\Vert
_{L^{\infty}}\leq C\left(  M,\Omega,\left\Vert m,\check{\varrho}%
,f,g\right\Vert _{L^{\infty}}\right)  .
\]
Thus, up to a subsequence, we obtain the existence of have that%
\begin{align*}
\varrho_{\varepsilon} &  \rightharpoonup\varrho\text{ weakly in }%
L^{\frac{r\gamma}{r-1}}(\Omega),\\
u_{\varepsilon} &  \rightharpoonup u\text{ weakly in }(W_{0}^{1,r}%
(\Omega))^{d},\\
\mathbb{P}_{\varepsilon}\left(  \mathbb{D}u_{\varepsilon}\right)   &
\rightharpoonup\mathbb{P}\text{ weakly-$\ast$ in }\left(  L^{\infty}(\Omega)\right)
^{d\times d}\\
\left\vert \mathbb{D}u_{\varepsilon}\right\vert ^{r-2}\mathbb{D}%
u_{\varepsilon} &  \rightharpoonup\mathbb{S}_{1}\text{ weakly in }(L^{\frac
{r}{r-1}}(\Omega))^{d\times d}.
\end{align*}
We denote
\[
\mathbb{S}=\mathbb{P}+\mathbb{S}_{1}\in(L^{\frac{r}{r-1}}(\Omega))^{d\times
d}.
\]
The importance of $\alpha>0$ lays in the fact that the equivalent of
Proposition \ref{propo_limit} in the present case is the following

\begin{proposition}
\label{Prop_identity_viscoplastic} The identity
\begin{equation}
\alpha\int_\Omega\overline{\varrho^{\gamma}}\varphi\;\dd x-\frac{1}{\gamma-1}\int_{\Omega}%
\overline{\varrho^{\gamma}}u\nabla\varphi \;\mathrm{d}x+\int_{\Omega}%
\overline{\varrho^{\gamma}}\mathrm{div\,}u\varphi \;\mathrm{d}x+\int_{\Omega
}\varphi\mathrm{d}\left[  \overline{\mathbb{S}:\mathbb{D}u}-\overline
{\mathbb{S}}:\mathbb{D}u\right]  =\alpha\int_\Omega\overline{\varrho^{\gamma-1}}%
\check{\varrho}\varphi\;\dd x
\label{identity_visco_plastic}%
\end{equation}
holds for all $\varphi\in C^{1}\left(  \overline{\Omega}\right)  $.
\end{proposition}

Substracting the renormalized continuity equation from $\left(
\text{\ref{identity_visco_plastic}}\right)  $ we also obtain that%
\begin{align*}
\alpha\int_\Omega\left(  \overline{\varrho^{\gamma}}-\varrho^{\gamma}\right)\varphi\;\dd x
-& \frac{1}{\gamma-1}\int_{\Omega}\left(  \overline{\varrho^{\gamma}}%
- \varrho^{\gamma}\right)  u\nabla\varphi \;\mathrm{d}x+ \int_{\Omega}\left(
\overline{\varrho^{\gamma}}-\varrho^{\gamma}\right)  \mathrm{div\,}%
u\varphi \;\mathrm{d}x \\
&+\int_{\Omega}\varphi\mathrm{d}\left[  \overline
{\mathbb{S}:\mathbb{D}u}-\overline{\mathbb{S}}:\mathbb{D}u\right] =\alpha\int_\Omega\left(  \overline{\varrho^{\gamma-1}}-\varrho^{\gamma-1}\right)
\check{\varrho}\varphi\;\dd x \text{.}%
\end{align*}
The restriction $\gamma\in(1,2]$ ensures that the RHS is negative. From
Proposition \ref{propo_limit}, the monotonicity of the dissipative operator
and Korn's inequality we obtain that
\begin{align*}
\varrho_{\varepsilon}  & \rightarrow\varrho\text{ in }L^{\gamma}\left(
\Omega\right)  ,\\
\nabla u_{\varepsilon}  & \rightarrow\nabla u\text{ strongly in }\left(
L^{q}\left(  \Omega\right)  \right)  ^{d},
\end{align*}
for any $q\in\lbrack1,r)$. Thus, up to a subsequence, we also have that for
all $i,j\in\overline{1,d}$%
\[
\partial_{i}u_{\varepsilon}^{j}\rightarrow\partial_{i}u^{j}\text{ a.e. on
}\Omega.
\]
In particular, this allows us to identify $\mathbb{S}_{1}=\left\vert
\mathbb{D}u\right\vert ^{r-2}\mathbb{D}u.$ From this and the definition
$\left(  \text{\ref{P_eps}}\right)  $ we see that%
\[
\lim_{\varepsilon\rightarrow0}\left\{  \mathbb{P}_{\varepsilon}+\left\vert
\mathbb{D}u_{\varepsilon}\right\vert ^{r-2}\mathbb{D}u_{\varepsilon}\right\}
=\tau^{\ast}\frac{\mathbb{D}u}{\left\vert \mathbb{D}u\right\vert }+\left\vert
\mathbb{D}u\right\vert ^{r-2}\mathbb{D}u\text{ on }\left\{  \left\vert
\mathbb{D}u\right\vert >0\right\}  .
\]
Next, since for all $\varphi\in L^{1}\left(  \Omega\right)$ with $\varphi\geq0$
\begin{align*}
0  & \leq\int_{\Omega}\left\vert \mathbb{P}_{\varepsilon}-\mathbb{P}%
\right\vert ^{2}\varphi\text{ }\mathrm{d}x=\int_{\Omega}\left(  \left\vert
\mathbb{P}\right\vert ^{2}-2\left\langle \mathbb{P}_{\varepsilon}%
,\mathbb{P}\right\rangle +\left\vert \mathbb{P}_{\varepsilon}\right\vert
^{2}\right)  \varphi\text{ }\mathrm{d}x\\
& \leq(\tau^{\ast})^{2}\int_{\Omega}\varphi\text{ }\mathrm{d}x+\int_{\Omega
}\left(  \left\vert \mathbb{P}\right\vert ^{2}-2\left\langle \mathbb{P}%
_{\varepsilon},\mathbb{P}\right\rangle \right)  \varphi\text{ }\mathrm{d}x
\end{align*}
passing to the limit we deduce that%
\[
\left\vert \mathbb{P}\right\vert \leq\tau^{\ast}%
\]
a.e. on $\Omega$ and thus%
\[
\left\vert \mathbb{S}\right\vert \leq\left\vert \mathbb{P}\right\vert
+\left\vert \mathbb{S}_{1}\right\vert \leq\tau^{\ast}+\left\vert
\mathbb{D}u\right\vert ^{r-1}=\tau^{\ast}\text{ on }\left\{  \left\vert
\mathbb{D}u\right\vert =0\right\}  .
\]

\paragraph{Acknowledgements:} The work of M. S. was supported by the National Science Centre grant no. 2022/45/N/ST1/03900 (Preludium). The work of C.B. has been partially supported by the project CRISIS (ANR-20-CE40-0020-01), operated by the French National Research Agency (ANR)

\begin{appendices}
\section{Construction of approximate solutions \label{CAS}}

In this section we prove Theorem \ref{Existence_4_levels} which concerns the
existence of solutions for the approximate system%
\begin{equation}
\left\{
\begin{array}
[c]{l}%
-\varepsilon\Delta\varrho+\eta\left(  \varrho-\dfrac{M}{\left\vert
\Omega\right\vert }\right)  +\operatorname{div}\left(  \varrho\omega_{\delta
}\ast T_{\delta}\left(  u\right)  \right)  =0,\\
-\operatorname{div}\mathbb{S}(\D u)-\alpha\operatorname{div}\left(  \left\vert
\nabla u\right\vert ^{q-2}\nabla u\right)  =F\left(  \varrho,u\right)  ,\\
\varrho\geq0,\text{ }%
{\displaystyle\int_{\Omega}}
\varrho\mathrm{d}x=M,\\
\frac{\partial\varrho}{\partial\overrightarrow{n}}=0\text{ },u=0\text{ on
}\partial\Omega
\end{array}
\right.  \label{Approximare_eps_delta_eta_alpha}%
\end{equation}
where%

\begin{align*}\label{definitie_F}
F\left(  \varrho,u\right)   &  =-\operatorname{div}\left(  \varrho\left(
\omega_{\delta}\ast T_{\delta}(u)\right)  \otimes u\right)  -T_{\delta}%
(\omega_{\delta}\ast\widetilde{\nabla\varrho^{\gamma}})\\
&  -\frac{\eta}{2}\varrho u-\varepsilon\nabla u\nabla\varrho+\varrho f+g.
\end{align*}

Recall that above we have
\[
\omega_{\delta}\left(  \cdot\right)  =\frac{1}{\delta^{3}}\omega\left(
\frac{1}{\delta}\cdot\right)  ,\text{ }T_{\delta}(v)\left(  x\right)
=\boldsymbol{1}_{\Omega_{2\delta}}\left(  x\right)  v\left(  x\right)  ,
\]
where $\omega\in\mathcal{D}\left(  \mathbb{R}^{3}\right)  $ a smooth,
non-negative, radial function which is compactly supported in the unit ball
centered at the origin and with integral $1$ while%
\[
\Omega_{2\delta}=\left\{  x\in\Omega:d\left(  x,\partial\Omega\right)
\geq2\delta\right\}  .
\]
By the $\widetilde{\cdot}$ we denote the trivial extension on $\mathbb{R}^{d}$
by $0$. Of course, we have that for any $u\in\left(  W^{1,q}\left(
\Omega\right)  \right)  ^{d},$ $\omega_{\delta}\ast T_{\delta}\left(
u\right)  \in W_{0}^{1,\infty}\left(  \Omega\right)  $. For the reader'
convenience, we recall below the precise statement of Theorem
\ref{Existence_4_levels_main}

\begin{theorem}
\label{Existence_4_levels}
Consider $\alpha,\delta,\varepsilon,\eta
\in\left(  0,1\right)  $, $d\in\left\{  2,3\right\}  ,$ $\gamma>1$,
$q>d$ and $r>\frac{d}{2}$. Let $\mathbb{S}:\mathbb{R}%
_{sym}^{d\times d}\rightarrow\mathbb{R}_{sym}^{d\times d}$ be a continuous
function verifying the conditions (\ref{growth})-(\ref{monotonicity}). For all $M>0$ and $f,g\in(L^{\infty}\left(  \Omega\right)  )^{d}%
$, there exists a solution $\left(  \varrho,u\right)  \in W^{2,2}\left(
\Omega\right)  \times(W_{0}^{1,q}\left(  \Omega\right)  )^{d}$ of $\left(  \text{\ref{Approximare_eps_delta_eta_alpha}%
}\right)  $ verifying the
identity
\begin{align}
\int_{\Omega}\varrho\text{ }\mathrm{d}x &  =M,\label{mass_conservation_appendix}\\
\int_{\Omega}\mathbb{S}(\D u) &  :\D u\text{ }\mathrm{d}x+\frac{4\varepsilon
}{\gamma}%
{\displaystyle\int_{\Omega}}
\left\vert \nabla\varrho^{\frac{\gamma}{2}}\right\vert ^{2}\mathrm{d}%
x+\alpha\int_{\Omega}\left\vert \nabla u\right\vert ^{q}\mathrm{d}%
x+\frac{\gamma\eta}{2\left(  \gamma-1\right)  }%
{\displaystyle\int_{\Omega}}
\varrho^{\gamma}\mathrm{d}x+\frac{\eta M}{\left\vert \Omega\right\vert }%
\int_{\Omega}\frac{\left\vert u\right\vert ^{2}}{2}\mathrm{d}x\nonumber\\
&  =\frac{\gamma\eta}{\left(  \gamma-1\right)  }\frac{M}{\left\vert
\Omega\right\vert }%
{\displaystyle\int_{\Omega}}
\varrho^{\gamma-1}\mathrm{d}x+\int_{\Omega}\left(  \varrho f+g\right)  u\text{
}\mathrm{d}x,\label{identity_1_appendix}%
\end{align}
along with the following estimates:%
\begin{equation}
\left\{
\begin{array}
[c]{l}%
\int_{\Omega}\mathbb{S(}u):\nabla u\text{ }\mathrm{d}x+\frac{4\varepsilon
}{\gamma}%
{\displaystyle\int_{\Omega}}
\left\vert \nabla\varrho^{\frac{\gamma}{2}}\right\vert ^{2}\mathrm{d}x\\
+\alpha\int_{\Omega}\left\vert \nabla u\right\vert ^{q}\mathrm{d}%
x+\frac{\gamma\eta}{2\left(  \gamma-1\right)  }%
{\displaystyle\int_{\Omega}}
\varrho^{\gamma}\mathrm{d}x+\frac{\eta M}{\left\vert \Omega\right\vert }%
\int_{\Omega}\frac{\left\vert u\right\vert ^{2}}{2}\mathrm{d}x\leq C\left(
M,\Omega,\varepsilon,\left\Vert f,g\right\Vert _{L^{\infty}}\right)  ,\\
\varepsilon\left\Vert \nabla^{2}\varrho\right\Vert _{L^{r}}\leq C\left(
M,\Omega,\varepsilon,\left\Vert f,g\right\Vert _{L^{\infty}}\right)  .
\end{array}
\right.  \label{estimari_eps_delta_prop}%
\end{equation}

\end{theorem}

The fact that we can solve the above system is a consequence of the
Leray-Schauder fixed point theorem, see Theorem \ref{Schauder_Leray} from the
Appendix B. Let us point out that the approximate system with the mollified
truncated velocity $\omega_{\delta}\ast T_{\delta}\left(  u\right)  $ that we
consider here is somehow non-standard and that a similar system was proposed
in \cite{BreschBurtea2021}. The rest of this section is devoted to the proof
of Theorem \ref{Existence_4_levels}.

\underline{Proof of Theorem \ref{Existence_4_levels}:}

Fix $\alpha,\delta,\varepsilon,\eta\in\left(  0,1\right)  ,$ $d\in\left\{
2,3\right\}  ,$ $\gamma>1$ and $q>d$. We begin by recalling the
following proposition which is taken from A. Novotn\'{y}-I.Stra\v{s}kraba
\cite{NovotnyStraskraba2004}, Proposition $4.29$ page $213$ :

\begin{proposition}
\label{transport_stationar}Let $\Omega$ be a domain of class $C^{2}$. Consider
$v\in\left(  W_{0}^{1,\infty}\left(  \Omega\right)  \right)  ^{d}$ and
$M,\delta,\varepsilon>0$. Then there exists an unique positive solution
$\varrho\in W^{2,2}\left(  \Omega\right)  $ for the equation%
\[
\left\{
\begin{array}
[c]{rr}%
-\varepsilon\Delta\varrho+\eta\left(  \varrho-\frac{M}{\left\vert
\Omega\right\vert }\right)  +\operatorname{div}\left(  \varrho v\right)  =0, &
\text{in }\Omega,\\
\frac{\partial\varrho}{\partial\overrightarrow{n}}\left(  x\right)  =0, &
\text{on }\partial\Omega.
\end{array}
\right.
\]
Moreover, there exists a positive constant $C\left(  M,\varepsilon\right)  $
depending on $\varepsilon$ and $M$ such that:%
\[
\left\{
\begin{array}
[c]{l}%
\varrho\geq0,\text{ }%
{\displaystyle\int_{\Omega}}
\varrho\mathrm{d}x=M,\\
\varepsilon%
{\displaystyle\int_{\Omega}}
\left\vert \nabla\varrho\right\vert ^{2}\mathrm{d}x+\eta%
{\displaystyle\int_{\Omega}}
\varrho^{2}\mathrm{d}x\leq C\left(  M,\varepsilon,\eta\right)  \left(
1+\left\Vert v\right\Vert _{W^{1,\infty}(\mathbb{T}^{3})}^{2}\right)  ,\\
\left\Vert \varrho\right\Vert _{W^{2,2}(\Omega)}\leq C\left(  M,\varepsilon
,\eta\right)  \left(  1+\left\Vert v\right\Vert _{W^{1,\infty}(\Omega)}%
^{2}\right)  .
\end{array}
\right.
\]

\end{proposition}

The above proposition is the starting point for the construction of solutions
for $\left(  \text{\ref{Approximare_eps_delta_eta_alpha}}\right)  $. These are
obtained as fixed points of the operator that is constructed as follows.
Consider
\begin{equation}
M>0,\qquad f,g\in(L^{\infty}\left(  \Omega\right)  )^{d}, \label{M_et_g}%
\end{equation}
along with $\alpha,\delta,\varepsilon,\eta\in\left(  0,1\right)  $, and for
any $v\in\left(  W_{0}^{1,r}\left(  \Omega\right)  \right)  ^{d},$ we consider
$\mathcal{E}\left(  v\right)  \in\left(  W_{0}^{1,q}\left(  \Omega\right)
\right)  ^{d}$ verifying
\begin{equation}
-\operatorname{div}\mathbb{S}(\D \mathcal{E}\left(  v\right)  )-\alpha
\operatorname{div}\left(  \left\vert \nabla\mathcal{E}\left(  v\right)
\right\vert ^{q-2}\nabla\mathcal{E}\left(  v\right)  \right)  =F\left(
\varrho,v\right)  , \label{definitie_E}%
\end{equation}
where $\varrho\in W^{2,2}\left(  \Omega\right)  $ is the unique solution of%
\begin{equation}
\left\{
\begin{array}
[c]{rr}%
-\varepsilon\Delta\varrho+\eta\left(  \varrho-\frac{M}{\left\vert
\Omega\right\vert }\right)  +\operatorname{div}\left(  \varrho\omega_{\delta
}\ast(T_{\delta}(v))\right)  =0, & \text{in }\Omega,\\
\frac{\partial\varrho}{\partial\overrightarrow{n}}\left(  x\right)  =0, &
\text{on }\partial\Omega.
\end{array}
\right.  \label{definitie_rho_E}%
\end{equation}
The existence of $\mathcal{E}\left(  v\right)  $ is a consequence of Browder's theorem, see Theorem
$1$ from \cite{Browder1963} and the fact that

\begin{proposition}
$F\left(  \varrho,v\right)  \in L^{\frac{q}{q-1}}\left(  \Omega\right)  $.
\end{proposition}

Recall that%
\begin{align*}
F\left(  \varrho,v\right)   &  =-\operatorname{div}\left(  \varrho\left(
\omega_{\delta}\ast T_{\delta}(v)\right)  \otimes v\right)  -T_{\delta}%
(\omega_{\delta}\ast\widetilde{\nabla\varrho^{\gamma}})\\
&  -\dfrac{\eta}{2}\varrho v-\varepsilon\nabla v\nabla\varrho+\varrho f+g.
\end{align*}

First, we write that%
\[
\left\Vert \varrho f+g\right\Vert _{L^{\infty}}\leq C\left(  M,\varepsilon
,\delta,\eta\right)  \left\Vert f\right\Vert _{L^{\infty}}\left\Vert
\varrho\right\Vert _{W^{2,2}}+\left\Vert g\right\Vert _{L^{\infty}}.
\]

Here we need $d\in\left\{  2,3\right\}  $ to insure that $L^{\infty}$ is imbedded
in $W^{2,2}\left(  \Omega\right)  $.

Next we write that
\[
\left\Vert \varrho v\right\Vert _{L^{\frac{q}{q-1}}}\leq C\left(
M,\varepsilon,\delta,\eta\right)  \left\Vert \varrho\right\Vert _{W^{2,2}%
}\left\Vert v\right\Vert _{L^{\frac{q}{q-1}}}\leq C\left(  \Omega
,M,\varepsilon,\delta,\eta\right)  \left\Vert \varrho\right\Vert _{W^{2,2}%
}\left\Vert v\right\Vert _{W^{1,q}}%
\]
where we use for instance $q\geq2$ for the last inequality.

Next, we observe that%
\[
\left\Vert \nabla v \nabla\varrho \right\Vert _{L^{\frac{q}{q-1}}}\leq\left\Vert
\nabla v\right\Vert _{L^{q}}\left\Vert \nabla\varrho\right\Vert _{\frac
{q}{q-2}}\leq C\left(  \Omega,M,\varepsilon,\delta,\eta\right)  \left\Vert
\varrho\right\Vert _{W^{2,2}}\left\Vert v\right\Vert _{W^{1,q}}.
\]
where we use%
\[
\frac{q}{q-2}\leq\frac{2d}{d-2}\Leftrightarrow1-\frac{2}{q}\geq\frac{1}%
{2}-\frac{1}{d}\Leftrightarrow\frac{1}{2}+\frac{1}{d}\geq\frac{2}{q},
\]
which is valid since $d\ge\frac{4d}{d+2}$

Writing%
\begin{align*}
\partial_{j}\left(  \varrho\left(  \omega_{\delta}\ast T_{\delta}(v^j)\right)
v^{i}\right)   &  =\partial_{j}\varrho\left(  \omega_{\delta}\ast T_{\delta
}(v^{j})\right)  v^{i}+\varrho\partial_{j}\left(  \omega_{\delta}\ast
T_{\delta}(v^{j})\right)  v^{i}\\
&  +\varrho\left(  \omega_{\delta}\ast T_{\delta}(v^{j})\right)  \partial
_{j}v^{i},
\end{align*}
we get that%
\[
\left\Vert \operatorname{div}\left(  \varrho\left(  \omega_{\delta}\ast
T_{\delta}(v)\right)  \otimes v\right)  \right\Vert _{L^{\frac{q}{q-1}}}\leq
C\left(  \Omega,M,\alpha,\varepsilon,\delta,\eta\right)  \left\Vert
\varrho\right\Vert _{W^{2,2}}\left\Vert v\right\Vert _{W^{1,q}}.
\]

Finally one has that%
\[
\left\Vert T_{\delta}(\omega_{\eta}\ast\widetilde{\nabla\varrho^{\gamma}%
})\right\Vert _{L^{\frac{q}{q-1}}}\leq C\left(  \Omega,M,\varepsilon
,\delta,\eta\right)  \left\Vert \varrho\right\Vert _{W^{2,2}}.
\]

It remains now to prove that we have a fixed point to solve the nonlinear
approximate system. This is the object of the following proposition.

\begin{proposition}
\label{Point_fix}The operator $\mathcal{E}$ defined by $\left(
\text{\ref{definitie_E}}\right)  $-$\left(  \text{\ref{definitie_rho_E}%
}\right)  $ admits a fixed point.
\end{proposition}

Proposition \ref{Point_fix} is a consequence of the Schauder-Leray theorem, see Theorem \ref{Schauder_Leray}. We will first prove that $\mathcal{E}$ is
continuous and compact and in a second time that the set%
\begin{equation}
\mathcal{P=}\left\{  u\in\left(  W_{0}^{1,q}\left(  \Omega\right)  \right)
^{d}:u=\lambda\mathcal{E}(u)\text{ for some }\lambda\in(0,1]\right\}
\label{set_P}%
\end{equation}
is bounded. \smallskip\noindent

\underline{ We begin by showing that the operator $\mathcal{E}$ is
continuous.}

First, let us recall that%
\[
-\varepsilon\Delta\varrho+\eta\left(  \varrho-\frac{M}{\left\vert
\Omega\right\vert }\right)  +\operatorname{div}\left(  \varrho\omega_{\delta
}\ast T_{\delta}(v)\right)  =0,
\]
then%
\begin{equation}
\left\Vert \varrho\right\Vert _{W^{2,2}}\leq C\left(  M,\Omega,\varepsilon
,\delta\right)  \left(  1+\left\Vert v\right\Vert _{W^{1,q}}^{2}\right)  ,
\label{Nouvelle_estimation_reg_rho}%
\end{equation}
see Proposition \ref{transport_stationar}.

Consider $R>0$ and $\left(  v_{0},v_{1}\right)  \in\left(  W_{0}^{1,q}\left(
\Omega\right)  \right)  ^{d}\times\left(  W_{0}^{1,q}\left(  \Omega\right)
\right)  ^{d}$ such that
\[
\left\Vert v_{1}-v_{0}\right\Vert _{W^{1,q}}\leq1\text{ and }\left\Vert
v_{1}\right\Vert _{W^{1,q}},\left\Vert v_{0}\right\Vert _{W^{1,q}}\leq R.
\]
Also, for $i\in\left\{  0,1\right\}  $ consider%
\[
-\varepsilon\Delta\varrho_{i}+\eta\left(  \varrho_{i}-\frac{M}{\left\vert
\Omega\right\vert }\right)  +\operatorname{div}\left(  \varrho_{i}%
\omega_{\delta}\ast T_{\delta}(v_{i})\right)  =0
\]
and%
\[
-\operatorname{div}\mathbb{S}(\D \mathcal{E}\left(  v_{i}\right)  )-\alpha
\operatorname{div}\left(  \left\vert \nabla\mathcal{E}\left(  v_{i}\right)
\right\vert ^{q-2}\nabla\mathcal{E}\left(  v_{i}\right)  \right)  =F\left(
\varrho_{i},v_{i}\right)  .
\]

First of all, the estimates $\left(  \text{\ref{Nouvelle_estimation_reg_rho}%
}\right)  $ allow us to conclude that%
\[
\left\Vert \varrho_{0}\right\Vert _{W^{2,2}}+\left\Vert \varrho_{1}\right\Vert
_{W^{2,2}}\leq C\left(  R,M,\Omega,\varepsilon,\delta,\eta\right)  .
\]
The difference $\left(  \varrho_{1}-\varrho_{0}\right)  $ verifies%
\begin{equation}
-\varepsilon\Delta\left(  \varrho_{1}-\varrho_{0}\right)  +\eta\left(  \left(
\varrho_{1}-\varrho_{0}\right)  \right)  +\operatorname{div}\left(
\varrho_{1}\omega_{\delta}\ast T_{\delta}(v_{1})-\varrho_{0}\omega_{\delta
}\ast T_{\delta}(v_{0})\right)  =0, \label{diferenta}%
\end{equation}
which provides the following estimate:%
\begin{equation}
\eta\int_{\Omega}\left\vert \varrho_{1}-\varrho_{0}\right\vert \leq
\int_{\Omega}\left\vert \operatorname{div}\left(  \varrho_{0}\left(
v_{1}-v_{0}\right)  \right)  \right\vert \leq\left\Vert \varrho_{0}\right\Vert
_{W^{1,2}}\left\Vert v_{1}-v_{0}\right\Vert _{W^{1,2}}\leq C(\Omega)\left\Vert
\varrho_{0}\right\Vert _{W^{1,2}}\left\Vert v_{1}-v_{0}\right\Vert _{W^{1,q}}.
\label{dependenta_norma_L1}%
\end{equation}
where we use that $q\geq2$. Next, we see that%
\begin{align}
\left\Vert \varrho_{1}-\varrho_{0}\right\Vert _{W^{1,2}}  &  \leq C\left\Vert
\varrho_{1}\omega_{\delta}\ast T_{\delta}(v_{1})-\varrho_{0}\omega_{\delta
}\ast T_{\delta}(v_{0})\right\Vert _{L^{2}}\nonumber\\
&  \leq C\left\Vert \omega_{\delta}\ast T_{\delta}(v_{1})\right\Vert
_{L^{\infty}}\left\Vert \varrho_{1}-\varrho_{0}\right\Vert _{L^{2}%
}+C\left\Vert \varrho_{0}\right\Vert _{L^{2}}\left\Vert \omega_{\delta}\ast
T_{\delta}(v_{1})-\omega_{\delta}\ast T_{\delta}(v_{0})\right\Vert
_{L^{\infty}}\nonumber\\
&  \leq C\left(  R,M,\Omega,\varepsilon,\delta,\eta\right)  \left(  \left\Vert
\varrho_{1}-\varrho_{0}\right\Vert _{L^{1}}^{\frac{1}{2}}+\left\Vert
v_{1}-v_{0}\right\Vert _{W^{1,q}}\right) \nonumber\\
&  \leq C\left(  R,M,\Omega,\varepsilon,\delta,\eta\right)  \left(  \left\Vert
v_{1}-v_{0}\right\Vert _{W^{1,q}}^{\frac{1}{2}}+\left\Vert v_{1}%
-v_{0}\right\Vert _{W^{1,q}}\right)  . \label{continuite_T_1}%
\end{align}
where we use $\left(  \text{\ref{dependenta_norma_L1}}\right)  $.

Next, by taking the difference of the velocity equations we end up with%
\begin{align}
&  -\operatorname{div}\left[  \mathbb{S}(\D\mathcal{E}\left(  v_{1}\right)
)-\mathbb{S}(\D \mathcal{E}\left(  v_{0}\right)  )\right]  -\alpha
\operatorname{div}\left(  \left\vert \nabla\mathcal{E}\left(  v_{1}\right)
\right\vert ^{q-2}\nabla\mathcal{E}\left(  v_{1}\right)  -\left\vert
\nabla\mathcal{E}\left(  v_{0}\right)  \right\vert ^{q-2}\nabla\mathcal{E}%
\left(  v_{0}\right)  \right) \nonumber\\
&  =F\left(  \varrho_{1},v_{1}\right)  -F\left(  \varrho_{0},v_{0}\right)  ,
\label{diff_v}%
\end{align}
where, we recall%
\begin{align*}
F\left(  \varrho,v\right)   &  =-\operatorname{div}\left(  \varrho\left(
\omega_{\delta}\ast T_{\delta}(v)\right)  \otimes v\right)  -T_{\delta}%
(\omega_{\delta}\ast\widetilde{\nabla\varrho^{\gamma}})\\
&  -\dfrac{\eta}{2}\varrho v-\varepsilon\nabla\varrho\nabla v+\varrho f+g.
\end{align*}
From $\left(  \text{\ref{diff_v}}\right)  $ we deduce that%
\begin{align*}
&  \left\Vert \nabla v_{1}-\nabla v_{0}\right\Vert _{W^{1,r}}^{r}%
+\alpha\left\Vert \nabla v_{1}-\nabla v_{0}\right\Vert _{W^{1,q}}^{q}\\
&  \leq C\left(  R,M,\Omega,\varepsilon,\delta,\eta,\alpha\right)  \left\{
\left\Vert \varrho_{1}v_{1}-\varrho_{0}v_{0}\right\Vert _{L^{1}}^{\frac
{q}{q-1}}+\left\Vert \nabla v_{1}\nabla\varrho_{1}-\nabla v_{0}\nabla
\varrho_{0}\right\Vert _{L^{1}}^{\frac{q}{q-1}}+\left\Vert f\left(
\varrho_{1}-\varrho_{0}\right)  \right\Vert _{L^{1}}^{\frac{q}{q-1}}\right. \\
&  \left.  \text{ \ \ \ \ \ \ \ \ \ \ \ \ \ \ \ \ \ \ \ \ \ \ \ \ \ \ \ \ }%
+\left\Vert T_{\delta}(\omega_{\delta}\ast\widetilde{\nabla\varrho_{1}%
^{\gamma}})-T_{\delta}(\omega_{\delta}\ast\widetilde{\nabla\varrho_{0}%
^{\gamma}}\right\Vert _{L^{1}}^{\frac{q}{q-1}}+\left\Vert \varrho_{1}%
\omega_{\delta}\ast T_{\delta}(v_{1})\otimes v_{1}-\varrho_{0}\omega_{\delta
}\ast T_{\delta}(v_{0})\otimes v_{0}\right\Vert _{L^{\frac{q}{q-1}}}^{\frac
{q}{q-1}}\right\}  .
\end{align*}
Using $\left(  \text{\ref{dependenta_norma_L1}}\right)  $, the first term is
treated as follows%
\begin{align}
\left\Vert \varrho_{0}v_{0}-\varrho_{1}v_{1}\right\Vert _{L^{1}}  &
\leq\left\Vert v_{0}\right\Vert _{L^{\infty}}\left\Vert \varrho_{1}%
-\varrho_{0}\right\Vert _{L^{1}}+\left\Vert \varrho_{1}\right\Vert _{L^{1}%
}\left\Vert v_{1}-v_{0}\right\Vert _{L^{\infty}}\nonumber\\
&  \leq C\left(  R,M,\Omega,\varepsilon,\delta\right)  \left\Vert v_{1}%
-v_{0}\right\Vert _{W^{1,q}}. \label{Term0}%
\end{align}
The second term is estimated as follows%
\begin{align}
&  \left\Vert \varrho_{1}\omega_{\delta}\ast T_{\delta}(v_{1})\otimes
v_{1}-\varrho_{0}\omega_{\delta}\ast T_{\delta}(v_{0})\otimes v_{0}\right\Vert
_{L^{\frac{q}{q-1}}}\nonumber\\
&  \leq\left\Vert \omega_{\delta}\ast T_{\delta}(v_{1})\otimes v_{1}%
\right\Vert _{L^{\infty}}\left\Vert \varrho_{1}-\varrho_{0}\right\Vert
_{L^{\frac{q}{q-1}}}+\left\Vert \varrho_{0}\right\Vert _{L^{\frac{q}{q-1}}%
}\left\Vert \omega_{\delta}\ast T_{\delta}(v_{1})\otimes v_{1}-\varrho
_{0}\omega_{\delta}\ast T_{\delta}(v_{0})\otimes v_{0}\right\Vert _{L^{\infty
}}\nonumber\\
&  \leq\left\Vert v_{1}\right\Vert _{L^{\infty}}^{2}\left\Vert \varrho
_{1}-\varrho_{0}\right\Vert _{L^{2}}\nonumber\\
&  +\left\Vert \varrho_{0}\right\Vert _{L^{\infty}}\left\{  \left\Vert
v_{1}\right\Vert _{L^{\infty}}\left\Vert \omega_{\delta}\ast
T_{\delta}(v_{1})-\omega_{\delta}\ast T_{\delta}(v_{0})\right\Vert
_{L^{\infty}}+\left\Vert \omega_{\delta}\ast T_{\delta}(v_{0})\right\Vert
_{L^{\infty}}\left\Vert v_{1}-v_{0}\right\Vert _{L^{\infty}}\right\}
\nonumber\\
&  \leq C\left(  R,M,\Omega,\varepsilon,\delta,\eta,\alpha\right)  \left(
\left\Vert v_{1}-v_{0}\right\Vert _{W^{1,q}}^{\frac{1}{2}}+\left\Vert
v_{1}-v_{0}\right\Vert _{W^{1,q}}\right)  . \label{Term1}%
\end{align}
The third term is treated using the Sobolev inequality along with $\left(
\text{\ref{Nouvelle_estimation_reg_rho}}\right)  $ and $\left(
\text{\ref{dependenta_norma_L1}}\right)  $%
\begin{align}
\left\Vert T_{\delta}(\omega_{\delta}\ast\widetilde{\nabla\varrho_{1}^{\gamma
}})-T_{\delta}(\omega_{\delta}\ast\widetilde{\nabla\varrho_{1}^{\gamma}%
})\right\Vert _{L^{1}}  &  \leq C\left(  \delta\right)  \left\Vert \varrho
_{1}^{\gamma}-\varrho_{1}^{\gamma}\right\Vert _{L^{1}}\nonumber\\
&  \leq C\left(  \delta\right)\left(  \left\Vert \varrho_{1}\right\Vert _{L^{\infty}}^{\gamma
-1}+\left\Vert \varrho_{0}\right\Vert _{L^{\infty}}^{\gamma-1}\right)
\left\Vert \varrho_{1}-\varrho_{0}\right\Vert _{L^{1}}\nonumber\\
&  \leq C\left(  R,M,\Omega,\varepsilon,\delta,\eta,\alpha\right)  \left\Vert
v_{1}-v_{0}\right\Vert _{W^{1,q}}. \label{Term2}%
\end{align}
The fourth term is treated with the help of relations $\left(
\text{\ref{Nouvelle_estimation_reg_rho}}\right)  $ and $\left(
\text{\ref{continuite_T_1}}\right)  $
\begin{align}
&  \left\Vert \nabla v_{1}\nabla\varrho_{1}-\nabla v_{0}\nabla\varrho
_{0}\right\Vert _{L^{1}}\nonumber\\
&  \leq\left\Vert \nabla v_{1}\right\Vert _{L^{2}}\left\Vert \nabla
(\varrho_{1}-\varrho_{0})\right\Vert _{L^{2}}+\left\Vert \nabla\varrho
_{1}\right\Vert _{L^{2}}\left\Vert \nabla v_{1}-\nabla v_{0}\right\Vert
_{L^{2}}\nonumber\\
&  \leq C\left(  R,M,\Omega,\alpha,\varepsilon,\delta,\eta\right)  \left(
\left\Vert v_{1}-v_{0}\right\Vert _{W^{1,q}}^{\frac{1}{2}}+\left\Vert
v_{1}-v_{0}\right\Vert _{W^{1,q}}\right)  . \label{Term3}%
\end{align}
From $\left(  \text{\ref{Term0}}\right)  $, $\left(  \text{\ref{Term1}%
}\right)  $, $\left(  \text{\ref{Term2}}\right)  $, $\left(  \text{\ref{Term3}%
}\right)  $ we obtain that $\mathcal{E}$ is continuous.

\underline{Let us show now that the operator $\mathcal{E}$ is compact.}
\medskip\noindent

Consider $R>0$ and a sequence $\left(  v_{n}\right)  _{n}\subset\left(
W_{0}^{1,q}\left(  \Omega\right)  \right)  ^{d}$ such that for all $n\geq0$ we
have
\[
\left\Vert v_{n}\right\Vert _{W^{1,q}}\leq R.
\]
We consider%
\[
-\varepsilon\Delta\varrho_{n}+\eta\left(  \varrho_{n}-\frac{M}{\left\vert
\Omega\right\vert }\right)  +\operatorname{div}\left(  \varrho_{n}%
\omega_{\delta}\ast T_{\delta}(v_{n})\right)  =0
\]
and $\left(  \mathcal{E}\left(  v_{n}\right)  \right)  _{n}\subset\left(
W_{0}^{1,q}\left(  \Omega\right)  \right)  ^{d}$ constructed as the solutions of
\begin{equation}
-\operatorname{div}\mathbb{S}(\D\mathcal{E}\left(  v_{n}\right)  )-\alpha
\operatorname{div}\left(  \left\vert \nabla\mathcal{E}\left(  v_{n}\right)
\right\vert ^{q-2}\nabla\mathcal{E}\left(  v_{n}\right)  \right)  =F\left(
\varrho_{n},v_{n}\right)  . \label{equation_vn}%
\end{equation}
In the follow, by $C$ we mean a generic constant that $C=C\left(
R,M,\Omega,\varepsilon,\delta,\eta\right)  $ independent of $n$. First of all,
we obtain that%
\[
\left\Vert \varrho_{n}\right\Vert _{W^{2,2}}\leq C.
\]
such that up to a subsequence that is still denoted $\left(  \varrho
_{n}\right)  _{n}$ by slightly abusing notation:
\begin{equation}
\left.
\begin{array}
[c]{r}%
\varrho_{n}\rightarrow\varrho\text{ strongly in }L^{\infty}\left(
\Omega\right)  ,\\
\\
\nabla\varrho_{n}\rightarrow\nabla\varrho_{n}\text{ strongly in }L^{m}\left(
\Omega\right)  \text{ for all }m\in\left\{
\begin{array}
[c]{c}%
\lbrack2,6)\text{ if }d=3,\\
\lbrack2,\infty)\text{ for }d=2
\end{array}
\right.  ,\\
\\
v_{n}\rightarrow v\text{ and }\mathcal{E}\left(  v_{n}\right)  \rightarrow
\overline{\mathcal{E}\left(  v\right)  }\text{ strongly in }(L^{\infty}\left(
\Omega\right)  )^{d},\\
\\
\nabla v_{n}\rightharpoonup\nabla v\text{ and }\nabla\mathcal{E}\left(
v_{n}\right)  \rightharpoonup\nabla\overline{\mathcal{E}\left(  v\right)
}\text{ weakly in }(W_{0}^{1,q}\left(  \Omega\right)  ^{d\times d},\\
\\
\mathbb{S(D}\mathcal{E}\left(  v_{n}\right)  )\rightharpoonup\overline
{\mathbb{S(D}\mathcal{E}\left(  v\right)  )},\left\vert \nabla\mathcal{E}%
\left(  v_{n}\right)  \right\vert ^{q-2}\nabla\mathcal{E}\left(  v_{n}\right)
\rightharpoonup\overline{\left\vert \nabla\mathcal{E}\left(  v\right)
\right\vert ^{q-2}\nabla\mathcal{E}\left(  v\right)  }\text{ weakly in }%
(W_{0}^{1,\frac{q}{q-1}}\left(  \Omega\right)  )^{d\times d},
\end{array}
\right.  \text{ }\label{list_conv_n}%
\end{equation}
Using theese relations, we obtain that%
\[
-\varepsilon\Delta\varrho+\eta\left(  \varrho-\frac{M}{\left\vert
\Omega\right\vert }\right)  +\operatorname{div}\left(  \varrho\omega_{\delta
}\ast T_{\delta}\left(  v\right)  \right)  =0
\]
along with%
\begin{equation}
-\operatorname{div}\overline{\mathbb{S(D}\mathcal{E}\left(  v\right)
)}-\alpha\operatorname{div}\left(  \overline{\left\vert \nabla\mathcal{E}%
\left(  v\right)  \right\vert ^{q-2}\nabla\mathcal{E}\left(  v\right)
}\right)  =F\left(  \varrho,v\right)  \label{limit_equation_appendix}%
\end{equation}
Testing $\left(  \text{\ref{equation_vn}}\right)  $ with $\mathcal{E}\left(
v_{n}\right)  $ we obtain that%
\begin{align*}
\int_{\Omega}\mathbb{S(D}\mathcal{E}\left(  v_{n}\right)  ) &  :\mathbb{D}%
\mathcal{E}\left(  v_{n}\right)  +\alpha\int_{\Omega}\left\vert \nabla
\mathcal{E}\left(  v_{n}\right)  \right\vert ^{q}\mathrm{d}x\\
&  =\int_{\Omega}\varrho_{n}\left(  \omega_{\delta}\ast T_{\delta}%
(v_{n})\right)  \otimes v_{n}:\nabla\mathcal{E}\left(  v_{n}\right)
\mathrm{d}x\\
&  +\frac{\eta}{2}\int_{\Omega}\varrho_{n}v_{n}\mathcal{E}\left(
v_{n}\right)  \mathrm{d}x+\varepsilon\int_{\Omega}\nabla\varrho_{n}\nabla
v_{n}\mathcal{E}\left(  v_{n}\right)  \mathrm{d}x+\int_{\Omega}\left(
\varrho_{n}f+g\right)  \mathcal{E}\left(  v_{n}\right)  \mathrm{d}x.
\end{align*}
Passing to the limit, and using $\left(  \text{\ref{list_conv_n}}\right)  $ we
find that%
\begin{align*}
&  \lim_{n\rightarrow\infty}\left\{  \int_{\Omega}\mathbb{S(}\D\mathcal{E}%
\left(  v_{n}\right)  ):\mathbb{D}\mathcal{E}\left(  v_{n}\right)
\mathrm{d}x+\alpha\int_{\Omega}\left\vert \nabla\mathcal{E}\left(
v_{n}\right)  \right\vert ^{q}\mathrm{d}x\right\}  \\
&  =\int_{\Omega}\varrho\left(  \omega_{\delta}\ast T_{\delta}(v)\right)
\otimes v:\nabla\overline{\mathcal{E}\left(  v\right)  }\mathrm{d}x\\
&  +\frac{\eta}{2}\int_{\Omega}\varrho v\overline{\mathcal{E}\left(  v\right)
}\mathrm{d}x+\varepsilon\int_{\Omega}\nabla\varrho\nabla v\overline
{\mathcal{E}\left(  v\right)  }\mathrm{d}x+\int_{\Omega}\left(  \varrho
f+g\right)  \overline{\mathcal{E}\left(  v\right)  }\mathrm{d}x\\
&  =\int_{\Omega}\overline{\mathbb{S(}\D\mathcal{E}\left(  v\right)
)}:\mathbb{D}\overline{\mathcal{E}\left(  v\right)  }\mathrm{d}x+\alpha
\int_{\Omega}\overline{\left\vert \nabla\mathcal{E}\left(  v\right)
\right\vert ^{q-2}\nabla\mathcal{E}\left(  v\right)  }:\nabla\overline
{\mathcal{E}\left(  v\right)  }\mathrm{d}x.
\end{align*}
Of course, the last equality comes from testing $\left(
\text{\ref{limit_equation_appendix}}\right)  $. Using the monotonicity of the
operator $\mathbb{S}\left(  \cdot\right)  \mathbb{+\alpha}\left\vert
\nabla\cdot\right\vert ^{q-2}\nabla\cdot$ we obtain that%
\[
\lim\limits_{n\rightarrow\infty}\mathcal{E}\left(  v_{n}\right)
=\overline{\mathcal{E}\left(  v\right)  }%
\]
in $(W^{1,q}(\Omega))^{d}$. The conclusion is that the operator $\mathcal{E}$
is compact.

\underline{Next, we prove that the set $\mathcal{P}$ defined in $\left(
\text{\ref{set_P}}\right)  $ is bounded.} Recall that%
\[
\mathcal{P=}\left\{  u\in(W_{0}^{1,q}\left(  \Omega\right)  )^{d}%
:u=\lambda\mathcal{E}\left(  u\right)  \text{ for some }\lambda\in
(0,1]\right\}  .
\]
Consider $\lambda\in(0,1]$ and $u\in(W_{0}^{1,q}\left(  \Omega\right)  )^{d}$
such that $u=\lambda\mathcal{E}\left(  u\right)  $. Obviously, one has%
\begin{equation}
\left\{
\begin{array}
[c]{l}%
-\varepsilon\Delta\varrho+\eta\left(  \varrho-\dfrac{M}{\left\vert
\Omega\right\vert }\right)  +\operatorname{div}\left(  \varrho\omega_{\delta
}\ast T_{\delta}(\lambda\mathcal{E}\left(  u\right)  )\right)  =0,\\
-\operatorname{div}\mathbb{S(}\mathcal{E}\left(  \mathbb{D}u\right)
)-\alpha\operatorname{div}\left(  \left\vert \nabla\mathcal{E}\left(
u\right)  \right\vert ^{q-2}\nabla\mathcal{E}\left(  u\right)  \right)
=F\left(  \varrho,\lambda\mathcal{E}\left(  u\right)  \right)  ,\\%
{\displaystyle\int_{\Omega}}
\varrho=M,\varrho\geq0,\text{ }\mathcal{E}\left(  u\right)  \in(W_{0}%
^{1,q}\left(  \Omega\right)  )^{d}.
\end{array}
\right.  \label{Leray-Schauder}%
\end{equation}

Observe that%
\begin{align}
&  \left\langle \operatorname{div}\left(  \varrho\omega_{\delta}\ast
T_{\delta}(\lambda\mathcal{E}\left(  u\right)  )\otimes\lambda\mathcal{E}%
\left(  u\right)  \right)  +\varepsilon\nabla\left(  \lambda\mathcal{E}\left(
u\right)  \right)  \nabla\varrho,\mathcal{E}\left(  u\right)  \right\rangle \\
&  =\frac{\lambda}{2}\operatorname{div}\left(  \varrho\omega_{\delta}\ast
T_{\delta}(\lambda\mathcal{E}\left(  u\right)  )\left\vert \mathcal{E}\left(
u\right)  \right\vert ^{2}\right)  +\operatorname{div}\left(  \varrho
\omega_{\delta}\ast T_{\delta}(\lambda\mathcal{E}\left(  u\right)  \right)
\frac{\lambda\left\vert \mathcal{E}\left(  u\right)  \right\vert ^{2}}%
{2}+\dfrac{\varepsilon\lambda}{2}\left\langle \nabla\left\vert \mathcal{E}%
\left(  u\right)  \right\vert ^{2},\nabla\varrho\right\rangle \nonumber\\
&  =\frac{\lambda}{2}\operatorname{div}\left(  \varrho\omega_{\delta}\ast
T_{\delta}(\lambda\mathcal{E}\left(  u\right)  )\left\vert \mathcal{E}\left(
u\right)  \right\vert ^{2}\right)  +\left(  \varepsilon\Delta\varrho
-\eta\left(  \varrho-\dfrac{M}{\left\vert \Omega\right\vert }\right)  \right)
\frac{\lambda\left\vert \mathcal{E}\left(  u\right)  \right\vert ^{2}}%
{2}+\dfrac{\varepsilon\lambda}{2}\left\langle \nabla\left\vert \mathcal{E}%
\left(  u\right)  \right\vert ^{2},\nabla\varrho\right\rangle \nonumber\\
&  =\frac{\lambda}{2}\operatorname{div}\left(  \varrho\omega_{\delta}\ast
T_{\delta}(\lambda\mathcal{E}\left(  u\right)  )\left\vert \mathcal{E}\left(
u\right)  \right\vert ^{2}\right)  +\dfrac{\varepsilon\lambda}{2}%
\operatorname{div}\left(  \left\vert \mathcal{E}\left(  u\right)  \right\vert
^{2}\nabla\varrho\right)  -\left(  \varrho-\dfrac{M}{\left\vert \Omega
\right\vert }\right)  \frac{\lambda\eta\left\vert \mathcal{E}\left(  u\right)
\right\vert ^{2}}{2}. \label{calcul_div(u_otimesu)}%
\end{align}
Next%
\begin{align*}
\int_{\Omega}\mathcal{E}\left(  u\right)  T_{\delta}(\omega_{\delta}%
\ast\widetilde{\nabla\varrho^{\gamma}})\mathrm{d}x  &  =-\int_{\Omega}%
\varrho^{\gamma}\operatorname{div}\omega_{\delta}\ast T_{\delta}%
(\mathcal{E}\left(  u\right)  )\mathrm{d}x\\
&  =\frac{4\varepsilon}{\lambda\gamma}\int_{\Omega}\left\vert \nabla
\varrho^{\frac{\gamma}{2}}\right\vert ^{2}\mathrm{d}x+\frac{\gamma\eta
}{\left(  \gamma-1\right)  \lambda}\left(  \int_{\Omega}\varrho^{\gamma
}\mathrm{d}x-\frac{M}{\left\vert \Omega\right\vert }\int_{\Omega}%
\varrho^{\gamma-1}\mathrm{d}x\right)  .
\end{align*}
Thus, we have that%
\begin{align}
\int_{\Omega}\mathbb{S}(\D \mathcal{E}\left(  u\right)  )  &  :\D\mathcal{E}\left(  u\right)  \mathrm{d}x+\alpha\int_{\Omega}\left\vert
\nabla\mathcal{E}\left(  u\right)  \right\vert ^{q}\mathrm{d}x+\frac
{4\varepsilon}{\lambda\gamma}\int_{\Omega}\left\vert \nabla\varrho
^{\frac{\gamma}{2}}\right\vert ^{2}\mathrm{d}x+\frac{\gamma\eta}{\left(
\gamma-1\right)  \lambda}\int_{\Omega}\varrho^{\gamma}\mathrm{d}x+\frac{\eta
M}{\left\vert \Omega\right\vert }\int_{\Omega}\frac{\left\vert \mathcal{E}%
\left(  u\right)  \right\vert ^{2}}{2}\mathrm{d}x\nonumber\\
&  =\frac{\gamma\eta}{\left(  \gamma-1\right)  \lambda}\frac{M}{\left\vert
\Omega\right\vert }\int_{\Omega}\varrho^{\gamma-1}\mathrm{d}x+\int_{\Omega
}\left(  \varrho f+g\right)  \mathcal{E}\left(  u\right)  \mathrm{d}x
\label{bound_vitesse_ind_eps_delta}%
\end{align}
We multiply by $\lambda^{q}$ the previous identity and since $\lambda
\mathcal{E}\left(  u\right)  =u$, $\lambda\leq1$ and some straightforward
interpolation inequalities, we find that
\[
\alpha\int_{\Omega}\left\vert \nabla u\right\vert ^{q}\mathrm{d}x\leq C
\]
by a constant $C$ that is independent of $\lambda$. The last estimate implies
that $\mathcal{P}$ is a bounded set of $\left(  W^{1,q}\left(  \Omega\right)
\right)  ^{d}$. Having proved that the operator $\mathcal{E}$ verifies the
hypothesis announced in Theorem \ref{Schauder_Leray} we conclude
that$\mathcal{\ E}$ admits a fixed point. This concludes the proof of
Proposition \ref{Point_fix}.

As an immediate consequence of Proposition \ref{Point_fix} we get the
existence part of Theorem \ref{Existence_4_levels} follows by observing that a
fixed point $u\in\left(  W_{0}^{1,q}\left(  \Omega\right)  \right)  ^{d}$ of
the operator $\mathcal{E}$ defined by $\left(  \text{\ref{definitie_E}%
}\right)  $ turns out to verify $\left(
\text{\ref{Approximare_eps_delta_eta_alpha}}\right)  $. The first estimate of
$\left(  \text{\ref{estimari_eps_delta_prop}}\right)  $ is nothing else but
$\left(  \text{\ref{bound_vitesse_ind_eps_delta}}\right)  $ with $\lambda=1$
followed by straightforward use of interpolation and H\"{o}lder inequalities.
The second estimate also follows from straightforward use of interpolation and
H\"{o}lder inequalities.

\section{Functional analysis tools\label{Section_tools}}

\subsection*{Functional analysis tools}

We begin with the classical

\begin{proposition}
[Sobolev-Poincar\'{e} inequality]Consider $p\in\lbrack1,d)$ and $g\in
W_{0}^{1,p}\left(  \Omega\right)  $ or $g\in W^{1,p}\left(  \Omega\right)  $
with $\int_{\Omega}g\mathrm{d}x=0$. Then,
\[
\left\Vert g\right\Vert _{L^{p_{\star}}}\leq\left\Vert \nabla g\right\Vert
_{L^{p}}%
\]
where $\frac{1}{p^{\star}}=\frac{1}{p}-\frac{1}{d}$.
\end{proposition}

This section is devoted to a quick recall of the main results from functional
analysis that we used thought the text. Consider $p\in\lbrack1,\infty)$, $g\in
L^{p}\left(  \mathbb{R}^{d}\right)  $ and $\omega\in\mathcal{D}\left(
\mathbb{R}^{3}\right)  $ a smooth, nonnegative, even function compactly
supported in the unit ball centered at the origin and with integral equal to
$1$. For all $\varepsilon>0$, we introduce the averaged functions
\begin{equation}
g_{\varepsilon}=g\ast\omega_{\varepsilon}(x)\qquad\hbox{ where }\qquad
\omega_{\varepsilon}=\frac{1}{\varepsilon^{3}}\omega(\frac{x}{\varepsilon}).
\label{notation_approx}%
\end{equation}
We recall the following classical analysis result
\[
\lim_{\varepsilon\rightarrow0}\left\Vert g_{\varepsilon}-g\right\Vert _{L^{p}%
}=0.
\]
Moreover, for any multi-index $\alpha$ there exists a constant $C\left(
\varepsilon,\alpha\right)  $ such that
\[
\left\Vert \partial^{\alpha}g_{\varepsilon}\right\Vert _{L^{\infty}}\leq
C\left(  \varepsilon,\alpha\right)  \left\Vert g\right\Vert _{L^{p}}.
\]
Next let us recall the following result concerning the commutator between the
convolution with $\omega_{\varepsilon}$ and the product with a given function.
More precisely, we have that

\begin{proposition}
\label{Prop_ren1}Consider $d\geq2,$ $1<q,\beta<\infty$, $\frac{1}{s}=\frac
{1}{\beta}+\frac{1}{q}\leq1$ and let $\left(  a,b\right)  $ be such that $a\in
L_{loc}^{\beta}\left(  \mathbb{R}^{d}\right)  $ and $b,\nabla b\in L_{loc}%
^{q}\left(  \mathbb{R}^{d}\right)  $. Then, we have%
\[
\lim r_{\varepsilon}\left(  a,b\right)  =0\text{ in }L_{loc}^{s}\left(
\mathbb{R}^{d}\right)
\]
where
\begin{equation}
r_{\varepsilon}\left(  a,b\right)  =\partial_{i}\left(  a_{\varepsilon
}b\right)  -\partial_{i}\left(  \left(  ab\right)  _{\varepsilon}\right)  ,
\label{def_reminder}%
\end{equation}
with $i\in\overline{1,d}$.
\end{proposition}

Next, we recall the definition and basic properties of the so-called Bogovskii
operator. We cite the precise form that can be found in
\cite{NovotnyStraskraba2004} Lemma $3.17,$ page $169.$

\begin{proposition}
\label{Bogo}Let $\Omega\subset\mathbb{R}^{d}$ a bounded Lipschitz domain.
There exists a linear operator $\mathcal{B}=\left(  \mathcal{B}^{1}%
,\mathcal{B}^{2},...\mathcal{B}^{d}\right)  $ with the following properties:%
\[
\left\{
\begin{array}
[c]{l}%
\mathcal{B}:\left\{  f\in L^{p}\left(  \Omega\right)  :\int_{\Omega}f\left(
x\right)  \mathrm{d}x=0\right\}  \rightarrow W_{0}^{1,p}\left(  \Omega\right)
,\text{ }1<p<\infty,\\
\operatorname{div}\mathcal{B}\left(  f\right)  =f\text{ a.e. in }\Omega,\\
\left\Vert \nabla\mathcal{B}\left(  f\right)  \right\Vert _{W_{0}^{1,p}}\leq
C\left(  p,\Omega\right)  \left\Vert f\right\Vert _{L^{p}}.
\end{array}
\right.
\]

\end{proposition}

We will also use see the following variant of Lemma $4.27$, page $211$ from
\cite{NovotnyStraskraba2004}:

\begin{proposition}
\label{caldura}Let $\Omega$ be a $C^{2}$ bounded domain, $p\in\left(
1,\infty\right)  ,M>0$ and $b\in\left(  L^{p}\left(  \Omega\right)  \right)
^{d}$. The equation
\[
\left\{
\begin{array}
[c]{c}%
-\varepsilon\Delta\varrho=\operatorname{div}b\text{ in }\Omega,\\
\partial_{n}\varrho=b\cdot\overrightarrow{n}\text{ on }\partial\Omega,\\
\int_{\Omega}\varrho\mathrm{d}x=M,
\end{array}
\right.
\]
admits an unique weak solution $\varrho\in W^{1,p}\left(  \Omega\right)  $ and
the estimate%
\[
\varepsilon\left\Vert \nabla\varrho\right\Vert _{L^{p}}\leq C\left(
p,\Omega\right)  \left\Vert b\right\Vert _{L^{p}}.
\]

Moreover, if $b\in E_{0}^{p}\left(  \Omega\right)  =\overline{\left(
\mathcal{D}\left(  \Omega\right)  \right)  ^{d}}^{\left\Vert \cdot\right\Vert
_{E^{p}}}$ where $\left\Vert b\right\Vert _{E^{p}}=\left\Vert b\right\Vert
_{L^{p}}+\left\Vert \operatorname{div}b\right\Vert _{L^{p}}$ then, moreover%
\[
\varepsilon\left\Vert \nabla^{2}\rho\right\Vert _{L^{p}}\leq C\left(
p,\Omega\right)  \left(  \left\Vert b\right\Vert _{L^{p}}+\left\Vert
\operatorname{div}b\right\Vert _{L^{p}}\right)  .
\]

\end{proposition}

Regarding the continuity equation we have the following

\begin{proposition}
\label{prolonged_continuity}Let $d\geq2$ and let $K$ a compact subset of a
bounded Lipschitz domain $\Omega\subset\mathbb{R}^{d}$. Let $\varrho\in
L^{(q^{\ast})^{\prime}}\left(  \Omega\right)  \cap L^{q^{\prime}}\left(
\Omega\backslash K\right)  $, $u\in(W_{0}^{1,q}\left(  \Omega\right)  )^{d}$
and $f\in L^{1}\left(  \Omega\right)  $. Suppose that%
\[
\operatorname{div}\left(  \varrho u\right)  =f\text{ in }\mathcal{D}^{\prime
}\left(  \Omega\right)  .
\]
Then, prolonging $\varrho,u$ and $f$ by $0$ outside $\Omega$ and, by slightly
abusing the notation, denoting the function again by $\varrho,u$ and $f$ we
have that%
\[
\operatorname{div}\left(  \varrho u\right)  =f\text{ in }\mathcal{D}^{\prime
}\left(  \mathbb{R}^{d}\right)  \text{.}%
\]

\end{proposition}

One also has the following:

\begin{proposition}
\label{Prop_ren2}Let $d\geq2$ and consider $2\leq\beta<\infty$ and
$\lambda_{0},\lambda_{1}$ such that
\[
0\leq1+\lambda_{1}\leq\dfrac{(r-1)\beta}{r}%
\]
Consider $\varrho\in L^{\beta}\left(  \mathbb{R}^{d}\right)  $, $\varrho\geq0$
a.e. and $u\in W^{1,r}\left(  \mathbb{R}^{d}\right)  $ verifying the steady
continuity equation%
\[
\operatorname{div}\left(  \varrho u\right)  =0
\]
in the sense of distributions. Then, for any function $b\in\mathcal{C}%
^{1}\left(  [0,\infty\right)  )$ such that%
\[
\left\vert b^{\prime}\left(  t\right)  \right\vert \leq ct^{\lambda_{1}}\text{
for }t\geq1
\]
it holds that%
\begin{equation}
\operatorname{div}\left(  b\left(  \varrho\right)  u\right)  +\left\{  \varrho
b^{\prime}\left(  \varrho\right)  -b\left(  \varrho\right)  \right\}
\operatorname{div}u=0. \label{renorm}%
\end{equation}
in the sense of distributions.
\end{proposition}

\noindent The proof of the above results follow by adapting in a
straightforward manner lemmas $3.1.,3.2.$ and $3.3.$ from the book of A.
Novotn\'{y}- I.Stra\v{s}kraba \cite{NovotnyStraskraba2004} pages $155-163$.

We end up this section with the following theorem that will be used to prove
existence of solutions:

\begin{theorem}
[Schauder-Leray]\label{Schauder_Leray}Let $\mathcal{T}$ be a continuous
compact mapping of a Banach space $\mathcal{B}$ into itself with the property
that there exists a real positive number $M>0$ such that%
\[
\left\Vert x\right\Vert _{\mathcal{B}}\leq M,
\]
for all $x$ such that $x=\lambda\mathcal{T}x$ for some $\lambda\in\left[
0,1\right]  $. Then $\mathcal{T}$ admits a fixed point.
\end{theorem}

\noindent For a proof of this result see Theorem $11.3.$ page $280$ from
\cite{GilbargTrudinger1977}.

\end{appendices}

\bibliographystyle{acm}
\bibliography{BibliotecaGenerala}

@Article{BreschBurtea2023,
  author    = {Bresch, Didier and Burtea, Cosmin},
  journal   = {Indiana University Mathematics Journal},
  title     = {Extension of the {H}off solutions framework to cover compressible {N}avier-{S}tokes equations with possible anisotropic viscous tensor},
  year      = {2023},
  issn      = {0022-2518},
  number    = {5},
  pages     = {2145--2189},
  volume    = {72},
  doi       = {10.1512/iumj.2023.72.9559},
  publisher = {Indiana University Mathematics Journal},
}

@Article{BreschBurtea2020,
  author       = {Bresch, Didier and Burtea, Cosmin},
  journal      = {Annales de l'Institut Henri Poincar{\'e} C, Analyse non lin{\'e}aire},
  title        = {Global existence of weak solutions for the anisotropic compressible {S}tokes system},
  year         = {2020},
  number       = {6},
  pages        = {1271-1297},
  volume       = {37},
  organization = {Elsevier},
}

@Article{BreschBurtea2021,
  author   = {Bresch, Didier and Burtea, Cosmin},
  journal  = {Journal de Mathématiques Pures et Appliquées},
  title    = {Weak solutions for the stationary anisotropic and nonlocal compressible {N}avier-{S}tokes system},
  year     = {2021},
  issn     = {0021-7824},
  pages    = {183-217},
  volume   = {146},
  doi      = {https://doi.org/10.1016/j.matpur.2020.07.013},
  url      = {https://www.sciencedirect.com/science/article/pii/S0021782420301306},
}

@Article{BreschJabin2018,
  author    = {Bresch, Didier and Jabin, Pierre--Emmanuel},
  journal   = {Annals of Mathematics},
  title     = {Global existence of weak solutions for compressible {N}avier-{S}tokes equations: thermodynamically unstable pressure and anisotropic viscous stress tensor},
  year      = {2018},
  number    = {2},
  pages     = {577--684},
  volume    = {188},
  publisher = {JSTOR},
}

@Book{Lions1998,
  author    = {Lions, Pierre-Louis},
  publisher = {The Clarendon Press, Oxford University Press, New York},
  title     = {Mathematical topics in fluid mechanics. {V}ol. 2},
  year      = {1998},
  isbn      = {0-19-851488-3},
  note      = {Compressible models, Oxford Science Publications},
  series    = {Oxford Lecture Series in Mathematics and its Applications},
  volume    = {10},
  mrclass   = {76-02 (35-02 35Q30 76N10)},
  pages     = {xiv+348},
}

@Article{Feireisl2001a,
  author    = {Feireisl, Eduard},
  journal   = {Commentationes Mathematicae Universitatis Carolinae},
  title     = {On compactness of solutions to the compressible isentropic {Navier-Stokes} equations when the density is not square integrable},
  year      = {2001},
  number    = {1},
  pages     = {83--98},
  volume    = {42},
  publisher = {Charles University in Prague, Faculty of Mathematics and Physics},
}

@Article{pokorny2022weak,
  author    = {Pokorn{\`y}, Milan and Szlenk, Maja},
  journal   = {Mathematical Methods in the Applied Sciences},
  title     = {Weak solutions for the {Stokes} system for compressible {non-Newtonian} fluids with unbounded divergence},
  year      = {2023},
  number    = {8},
  pages     = {9736--9750},
  volume    = {46},
  publisher = {Wiley Online Library},
}

@Article{feireisl2015global,
  author    = {Feireisl, Eduard and Liao, X and Malek, Josef},
  journal   = {Mathematical Methods in the Applied Sciences},
  title     = {Global weak solutions to a class of non-{N}ewtonian compressible fluids},
  year      = {2015},
  number    = {16},
  pages     = {3482--3494},
  volume    = {38},
  publisher = {Wiley Online Library},
}

@Article{bulivcek2019uniqueness,
  author    = {Bul{\'\i}{\v{c}}ek, Miroslav and Kaplick{\`y}, Petr and Pra{\v{z}}{\'a}k, Dalibor},
  journal   = {Mathematical Models and Methods in Applied Sciences},
  title     = {Uniqueness and regularity of flows of non-{N}ewtonian fluids with critical power-law growth},
  year      = {2019},
  number    = {06},
  pages     = {1207--1225},
  volume    = {29},
  publisher = {World Scientific},
}

@Article{Browder1963,
  author  = {Browder, Felix E},
  journal = {Bulletin of the American Mathematical Society},
  title   = {Nonlinear elliptic boundary value problems},
  year    = {1963},
  number  = {6},
  pages   = {862--874},
  volume  = {69},
}

@Book{NovotnyStraskraba2004,
  author    = {Novotn\'{y}, Antonin and Stra\v{s}kraba, Ivan},
  publisher = {OUP Oxford},
  title     = {Introduction to the mathematical theory of compressible flow},
  year      = {2004},
  volume    = {27},
}

@Book{GilbargTrudinger1977,
  author    = {Gilbarg, David and Trudinger, Neil S and Gilbarg, David and Trudinger, NS},
  publisher = {Springer},
  title     = {Elliptic partial differential equations of second order},
  year      = {1977},
  number    = {2},
  volume    = {224},
}

@Article{Mamontov2007,
  author    = {Mamontov, Alexander Evgenievich},
  journal   = {Mathematical Notes},
  title     = {Existence of global solutions to multidimensional equations for {Bingham} fluids},
  year      = {2007},
  pages     = {501--517},
  volume    = {82},
  publisher = {Springer},
}

@Article{Mamontov1999,
  author    = {Mamontov, AE},
  journal   = {Siberian Mathematical Journal},
  title     = {Global solvability of the multidimensional Navier-Stokes equations of a compressible fluid with nonlinear viscosity. I},
  year      = {1999},
  number    = {2},
  pages     = {351--362},
  volume    = {40},
  publisher = {Springer},
}

@Article{Mamontov1999_2,
  author    = {Mamontov, AE},
  journal   = {Siberian Mathematical Journal},
  title     = {Global solvability of the multidimensional Navier-Stokes equations of a compressible fluid with nonlinear viscosity. II},
  year      = {1999},
  number    = {3},
  pages     = {541-555},
  volume    = {40},
  publisher = {Springer},
}

@Article{MalekRajagopal2010,
  author    = {M{\'a}lek, Josef and Rajagopal, KR},
  journal   = {Zeitschrift f{\"u}r angewandte Mathematik und Physik},
  title     = {Compressible generalized {Newtonian} fluids},
  year      = {2010},
  number    = {6},
  pages     = {1097--1110},
  volume    = {61},
  publisher = {Springer},
}

@Article{Ladyzhenskaya1967,
  author    = {Ladyzhenskaya, Olga Aleksandrovna},
  journal   = {Trudy Matematicheskogo Instituta Imeni VA Steklova},
  title     = {New equations for the description of the motions of viscous incompressible fluids, and global solvability for their boundary value problems},
  year      = {1967},
  pages     = {85--104},
  volume    = {102},
  publisher = {Russian Academy of Sciences, Steklov Mathematical Institute of Russian~…},
}

@Article{Diening_et_al_2010,
  author  = {Diening, Lars and R{{\u{u}}}{{\v{z}}}i{{\v{c}}}ka, Michael and Wolf, J{\"o}rg},
  journal = {Annali della Scuola Normale Superiore di Pisa-Classe di Scienze},
  title   = {Existence of weak solutions for unsteady motions of generalized {Newtonian} fluids},
  year    = {2010},
  number  = {1},
  pages   = {1--46},
  volume  = {9},
}

@Article{bulicek2012unsteady,
  author    = {Bulicek, Miroslav and Gwiazda, Piotr and M{\'a}lek, Josef and Swierczewska-Gwiazda, Agnieszka},
  journal   = {SIAM Journal on Mathematical Analysis},
  title     = {On unsteady flows of implicitly constituted incompressible fluids},
  year      = {2012},
  number    = {4},
  pages     = {2756--2801},
  volume    = {44},
  publisher = {SIAM},
}

@Article{malek2005herschel,
  author    = {M{\'a}lek, Josef and R{{\u{u}}}{{\v{z}}}i{{\v{c}}}ka, Michael and Shelukhin, VV},
  journal   = {Mathematical Models and Methods in Applied Sciences},
  title     = {{Herschel--Bulkley} fluids: existence and regularity of steady flows},
  year      = {2005},
  number    = {12},
  pages     = {1845--1861},
  volume    = {15},
  publisher = {World Scientific},
}

@Article{BasovShelukhin1999,
  author    = {Basov, IV and Shelukhin, VV},
  journal   = {ZAMM-Journal of Applied Mathematics and Mechanics/Zeitschrift f{\"u}r Angewandte Mathematik und Mechanik: Applied Mathematics and Mechanics},
  title     = {{Generalized solutions to the equations of compressible Bingham flows}},
  year      = {1999},
  number    = {3},
  pages     = {185--192},
  volume    = {79},
  publisher = {Wiley Online Library},
}

@Article{Abbatiello_et_al2020,
  author      = {Abbatiello, Anna and Feireisl, Eduard and Novotny, Antonin},
  journal     = {{Discrete and Continuous Dynamical Systems - Series A}},
  title       = {{Generelazed solutions to models of compressible viscous fluids}},
  year        = {2020},
  hal_id      = {hal-02954153},
  hal_version = {v1},
  publisher   = {{American Institute of Mathematical Sciences}},
  url         = {https://univ-tln.hal.science/hal-02954153},
}

@Article{Basaric2022,
  author    = {Basari{\'c}, Danica},
  journal   = {Journal of Mathematical Fluid Mechanics},
  title     = {Existence of dissipative (and weak) solutions for models of general compressible viscous fluids with linear pressure},
  year      = {2022},
  number    = {2},
  pages     = {56},
  volume    = {24},
  publisher = {Springer},
}

@Article{Gwiazda2010,
  author    = {Gwiazda, Piotr and {\'S}wierczewska-Gwiazda, Agnieszka and Wr{\'o}blewska, Aneta},
  journal   = {Mathematical methods in the applied sciences},
  title     = {{Monotonicity methods in generalized Orlicz spaces for a class of non-Newtonian fluids}},
  year      = {2010},
  number    = {2},
  pages     = {125--137},
  volume    = {33},
  publisher = {Wiley Online Library},
}

@InProceedings{ZhikovPastukhova2009,
  author       = {Zhikov, VV and Pastukhova, SE},
  booktitle    = {Doklady Mathematics},
  title        = {On the solvability of the {Navier-Stokes} equations for a compressible {non-Newtonian} fluid},
  year         = {2009},
  organization = {Springer},
  pages        = {511--515},
  volume       = {80},
}

@Article{Burczak_et_al2021,
  author    = {Burczak, Jan and Modena, Stefano and Sz{\'e}kelyhidi, L{\'a}szl{\'o}},
  journal   = {Communications in Mathematical Physics},
  title     = {Non uniqueness of power-law flows},
  year      = {2021},
  pages     = {199--243},
  volume    = {388},
  publisher = {Springer},
}

@Book{duvant2012inequalities,
  author    = {Duvant, Georges and Lions, Jacques Louis},
  publisher = {Springer Science \& Business Media},
  title     = {Inequalities in mechanics and physics},
  year      = {2012},
  volume    = {219},
}

@Article{eberlein2012existence,
  author    = {Eberlein, Hannes and R{{\u{u}}}{{\v{z}}}i{{\v{c}}}ka, Michael},
  journal   = {Journal of Mathematical Fluid Mechanics},
  title     = {Existence of weak solutions for unsteady motions of {Herschel--Bulkley} fluids},
  year      = {2012},
  number    = {3},
  pages     = {485--500},
  volume    = {14},
  publisher = {Springer},
}

@Article{bulivcek2023non,
  author    = {Bul{\'\i}{\v{c}}ek, Miroslav and Gwiazda, Piotr and Skrzeczkowski, Jakub and Wo{\'z}nicki, Jakub},
  journal   = {Journal of Functional Analysis},
  title     = {Non-{Newtonian} fluids with discontinuous-in-time stress tensor},
  year      = {2023},
  number    = {2},
  pages     = {109943},
  volume    = {285},
  publisher = {Elsevier},
}

@Article{bulivcek2023unsteady,
  author    = {Bul{\'\i}{\v{c}}ek, Miroslav and M{\'a}lek, Josef and Maringov{\'a}, Erika},
  journal   = {Journal of Mathematical Fluid Mechanics},
  title     = {On Unsteady Internal Flows of Incompressible Fluids Characterized by Implicit Constitutive Equations in the Bulk and on the Boundary},
  year      = {2023},
  issn      = {1422-6952},
  month     = jul,
  number    = {3},
  volume    = {25},
  doi       = {DOI: 10.1007/s00021-023-00803-w},
  publisher = {Springer Science and Business Media LLC},
}

@Book{ladyzhenskaya1969mathematical,
  author  = {Ladyzhenskaya, Olga Aleksandrovna},
  title   = {The mathematical theory of viscous incompressible flow},
  year    = {1969},
  journal = {Gordon \& Breach},
}

@Book{galdi2011introduction,
  author    = {Galdi, Giovanni},
  publisher = {Springer Science \& Business Media},
  title     = {An introduction to the mathematical theory of the Navier-Stokes equations: Steady-state problems},
  year      = {2011},
}

@Article{Malek_et_al1993,
  author    = {M\'alek, Josef and Ne\v{c}as, Jind\v{r}ich and R{{\u{u}}}{{\v{z}}}i{{\v{c}}}ka, Michael},
  journal   = {Mathematical Models and Methods in Applied Sciences},
  title     = {On the {N}on-{N}ewtonian incompressible fluids},
  year      = {1993},
  issn      = {1793-6314},
  month     = feb,
  number    = {01},
  pages     = {35--63},
  volume    = {03},
  doi       = {https://doi.org/10.1142/S0218202593000047},
  publisher = {World Scientific Pub Co Pte Lt},
}

@Article{wolf2007existence,
  author    = {Wolf, J{\"o}rg},
  journal   = {Journal of Mathematical Fluid Mechanics},
  title     = {{Existence of weak solutions to the equations of non-stationary motion of non-Newtonian fluids with shear rate dependent viscosity}},
  year      = {2007},
  number    = {1},
  pages     = {104--138},
  volume    = {9},
  publisher = {Springer},
}

@Article{frehse2003analysis,
  author    = {Frehse, Jens and M{\'a}lek, Josef and Steinhauer, Mark},
  journal   = {SIAM journal on mathematical analysis},
  title     = {On analysis of steady flows of fluids with shear-dependent viscosity based on the {Lipschitz} truncation method},
  year      = {2003},
  number    = {5},
  pages     = {1064--1083},
  volume    = {34},
  publisher = {SIAM},
}

@Article{blechta2020classification,
  author    = {Blechta, Jan and M{\'a}lek, Josef and Rajagopal, KR},
  journal   = {SIAM Journal on Mathematical Analysis},
  title     = {On the classification of incompressible fluids and a mathematical analysis of the equations that govern their motion},
  year      = {2020},
  number    = {2},
  pages     = {1232--1289},
  volume    = {52},
  publisher = {SIAM},
}

@Article{AbbatielloBulicekKaplicky2022,
  author    = {Abbatiello, Anna and Bulíček, Miroslav and Kaplický, Petr},
  journal   = {Philosophical Transactions of the Royal Society A: Mathematical, Physical and Engineering Sciences},
  title     = {{On solutions for a generalized Navier–Stokes–Fourier system fulfilling the entropy equality}},
  year      = {2022},
  issn      = {1471-2962},
  month     = sep,
  number    = {2236},
  volume    = {380},
  doi       = {https://doi.org/10.1098/rsta.2021.0351},
  publisher = {The Royal Society},
}

@Article{Frehse2010,
  author    = {Frehse, Jens and M{\'a}lek, Josef and R{{\u{u}}}{{\v{z}}}i{{\v{c}}}ka, Michael},
  journal   = {Communications in Partial Differential Equations},
  title     = {Large data existence result for unsteady flows of inhomogeneous shear-thickening heat-conducting incompressible fluids},
  year      = {2010},
  number    = {10},
  pages     = {1891--1919},
  volume    = {35},
  publisher = {Taylor \& Francis},
}

@Article{FeireislLuMalek2015,
  author    = {Feireisl, Eduard and Lu, Yong and Málek, Josef},
  journal   = {ZAMM - Journal of Applied Mathematics and Mechanics / Zeitschrift für Angewandte Mathematik und Mechanik},
  title     = {On {PDE} analysis of flows of quasi-incompressible fluids},
  year      = {2015},
  issn      = {0044-2267},
  month     = apr,
  number    = {4},
  pages     = {491--508},
  volume    = {96},
  doi       = {https://doi.org/10.1002/zamm.201400229},
  publisher = {Wiley},
}

@Article{wroblewska2012unsteady,
  author    = {Wr{\'o}blewska-Kami{\'n}ska, Aneta},
  journal   = {Discrete and Continuous Dynamical Systems},
  title     = {{Unsteady flows of non-Newtonian fluids in generalized Orlicz spaces}},
  year      = {2012},
  number    = {6},
  pages     = {2565--2592},
  volume    = {33},
  publisher = {Discrete and Continuous Dynamical Systems},
}

@Article{Kazhikhov_Mamontov_1998,
  author    = {Kazhikhov, A. V. and Mamontov, A. E.},
  journal   = {Siberian Mathematical Journal},
  title     = {{On a certain class of convex functions and the exact well-posedness classes of the Cauchy problem for the transport equation in Orlicz spaces}},
  year      = {1998},
  issn      = {1573-9260},
  month     = aug,
  number    = {4},
  pages     = {716--734},
  volume    = {39},
  doi       = {10.1007/bf02673055},
  publisher = {Springer Science and Business Media LLC},
}

\end{document}